\documentclass[12pt,reqno,twoside]{amsart}
\usepackage{amsfonts,amsmath,amssymb}
\usepackage{mathrsfs,mathtools}
\usepackage{enumerate}
\usepackage{hyperref}
\usepackage{esint}
\usepackage{graphicx,subfigure}
\usepackage{bm}
\usepackage{commath}
\usepackage{esint}
\usepackage{placeins}
\usepackage{flafter}
\usepackage{tikz}
\usepackage[textsize=small]{todonotes}
\usepackage[dvips,bottom=1.4in,right=1in,top=1in, left=1in]{geometry}
\usepackage[latin1]{inputenc}

\usepackage{caption}

\DeclareMathAlphabet{\mathpzc}{OT1}{pzc}{m}{it}
\numberwithin{equation}{section}

\makeatletter
\def\eqnarray{\stepcounter{equation}\let\@currentlabel=\theequation
\global\@eqnswtrue
\tabskip\@centering\let\\=\@eqncr
$$\halign to \displaywidth\bgroup\hfil\global\@eqcnt\z@
  $\displaystyle\tabskip\z@{##}$&\global\@eqcnt\@ne
  \hfil$\displaystyle{{}##{}}$\hfil
  &\global\@eqcnt\tw@ $\displaystyle{##}$\hfil
  \tabskip\@centering&\llap{##}\tabskip\z@\cr}
\def\endeqnarray{\@@eqncr\egroup
      \global\advance\c@equation\m@ne$$\global\@ignoretrue}

\newtheorem{theorem}{Theorem}[section]

\newtheorem{definition}[theorem]{Def{}inition}
\newtheorem{example}[theorem]{Example}
\newtheorem{lemma}[theorem]{Lemma}

\newtheorem{proposition}[theorem]{Proposition}

\newtheorem{remark}[theorem]{Remark}
\numberwithin{equation}{section}

 \usepackage[textsize=small]{todonotes}
\newcommand{\R}{\mathbb{R}}

\DeclareMathOperator{\diag}{diag}

\title{An Optimal Time Variable Learning Framework for Deep Neural Networks}
\date{\today}

\thanks{
This work is partially supported by NSF grants DMS-2110263, DMS-1913004, DMS-2111315, the Air Force Office of Scientific Research (AFOSR) under Award NO: FA9550-19-1-0036, and Department of Navy, Naval PostGraduate School under Award NO: N00244-20-1-0005.
}

\author{Harbir Antil}
\address{H. Antil. The Center for Mathematics and Artificial Intelligence
(CMAI) and Department of Mathematical Sciences, George Mason University,
Fairfax, VA 22030, USA.}
\email{hantil@gmu.edu}
\author{Hugo D\'iaz}
\address{Hugo D\'iaz. Department of Mathematical Sciences, University of Delaware, Newark, DE 19716, USA.}
\email{hugodiaz@udel.edu}
\author{Evelyn Herberg}
\address{E. Herberg. The Center for Mathematics and Artificial Intelligence
(CMAI) and Department of Mathematical Sciences, George Mason University,
Fairfax, VA 22030, USA.}
\email{eherberg@gmu.edu}

\begin{document}

\begin{abstract}	
Feature propagation in Deep Neural Networks (DNNs) can be associated to nonlinear discrete dynamical systems. The novelty, in this paper, lies in letting the discretization parameter (time step-size) vary from layer to layer, which needs to be learned, in an optimization framework. The proposed framework can be applied to any of the existing networks such as ResNet, DenseNet or Fractional-DNN. This framework is shown to help overcome the vanishing and exploding gradient issues. Stability of some of the existing continuous DNNs such as Fractional-DNN is also studied. The proposed approach is applied to an ill-posed 3D-Maxwell's equation. 
\end{abstract}


\maketitle

\tableofcontents
	
\section{Introduction} \label{sec:Intro}


\noindent Consider a network architecture, for example, residual neural network (ResNet)
\begin{equation}
\label{eq:yintro}
	y^{[\ell]} = y^{[\ell -1]} + \tau \sigma (y^{[\ell-1]}, \theta^{[\ell-1]}) ,
\end{equation}
which contains a parameter $\tau$. It is also possible to consider other architectures such 
as feedforward networks etc. The above neural network can be understood as the 
time-discretization of a non-linear ordinary differential equation (ODE). The feature vector 
$y^{[\ell]}$ is computed by forward propagation from the previous layers feature vector 
$y^{[\ell-1]}$ and network parameters, which are collected in $\theta^{[\ell-1]}$, using an 
activation function $\sigma$. In most of the existing literature, $\tau$ is a given fixed 
constant. The main novelty of this work lies in 
\begin{center}
	replacing $\tau$ by learning variables $\tau^{[\ell]}$
\end{center}
and the treatment of these $\tau^{[\ell]}$. These variables can be understood as 
the `time step-sizes' in Deep Neural Networks (DNNs). 
This paper considers them as optimization variables, not as hyperparameters. Furthermore, the parameters 
$\tau^{[\ell]}$ are allowed to differ from layer to layer, i.e., time grid can be non-equidistant.
Notice that this $\tau$-variable framework can be applied to any of the existing networks of type 
\eqref{eq:yintro}. The deep learning optimization problem will now also learn optimal parameters 
$\tau^{[\ell]}$ in addition to the standard DNN parameters. 
As will be illustrated throughout the article,  the presented approach is not just a 
scaling of the activation function $\sigma$ by $\tau^{[\ell]}$. This will become 
evident when applying the proposed framework to Fractional-DNNs, where 
$\tau^{[\ell]}$ enters in multiple ways, see Remark~\ref{rem:scale}. 

\begin{figure}[t] 
	\includegraphics[width=0.83\textwidth]{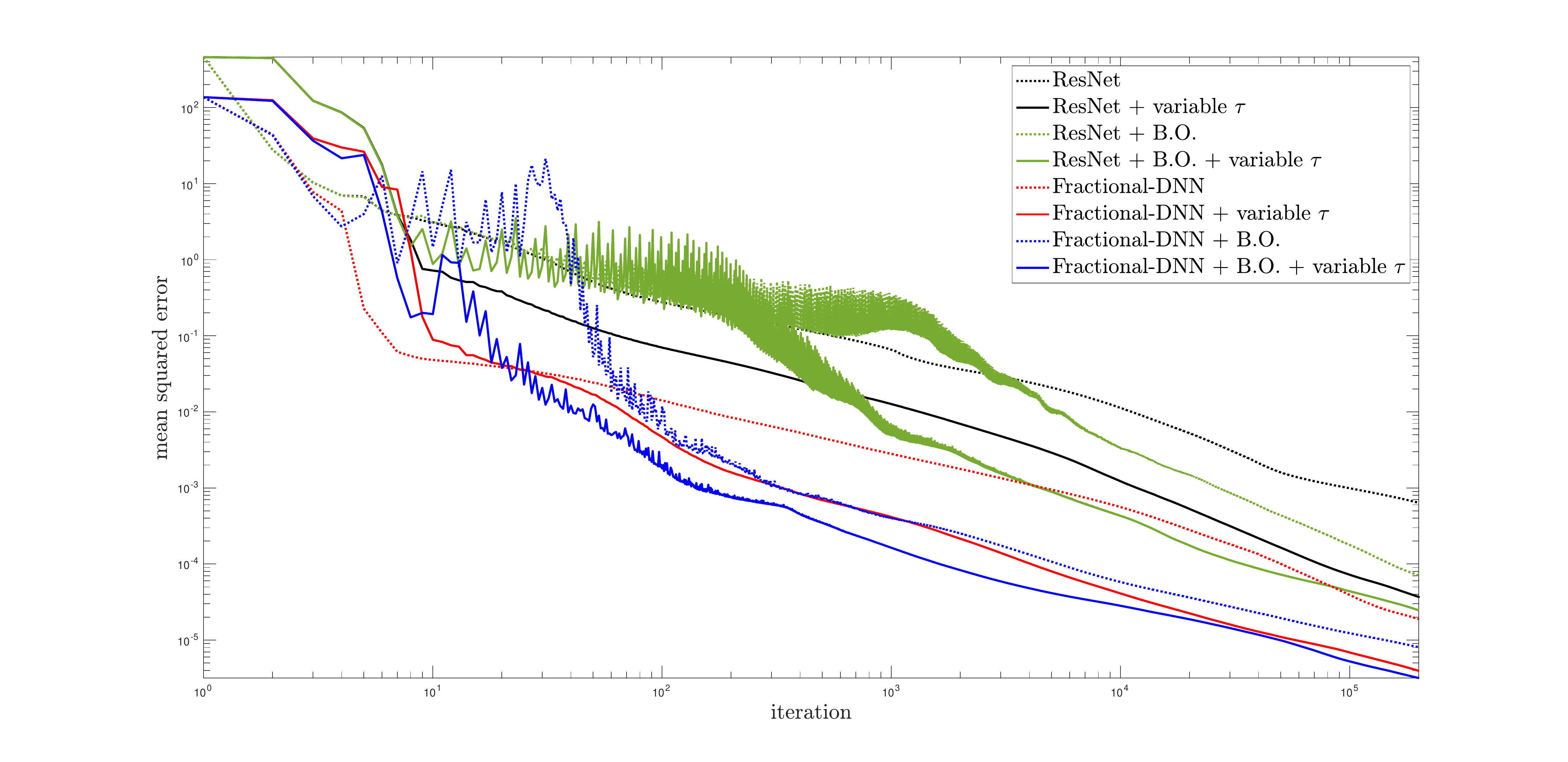}
	\caption{
	The panel shows the mean squared error during training when the variable-$\tau$ 
	framework is applied to a ResNet and a Fractional~DNN for an ill-posed 3D-Maxwell's 
	equation. More details are available in Section~\ref{sec:Num}.
	}
	\label{fig:mse_6_50}
\end{figure}
This proposed approach presents multiple advantages, including: 
\begin{itemize}
	\item The proposed framework leads to \emph{optimal adaptive time discretizations of DNNs}, such as ResNets 
		and Fractional-DNNs, tailored to the optimization/learning problem.
	\item The proposed framework can further help (as is rigorously established) overcome the \emph{vanishing and 
		exploding gradient} problems in networks such as ResNets and Fractional-DNNs. Notice, that motivation 
		behind introducing ResNets
		\cite{he2016deep} was vanishing gradients and Fractional-DNNs was vanishing and exploding gradients 
		\cite{antil2020fractional}.
	\item If $\tau^{[\ell]}$ for any layer $\ell$ is close to zero, then the associated layer is redundant and can be 
		deleted without sacrificing the accuracy. Thus leading to small yet accurate DNNs. 
	\item Variable $\tau^{[\ell]}$ helps improve the training error decay, see Figure~\ref{fig:mse_6_50}.
\end{itemize}
The main idea to approximate parameterized PDEs and solve the inverse problems using Fractional-DNNs has been recently introduced in \cite{antil2021novel}. The present article not only introduces the aforementioned variable time step framework, but for the first time, to the best of our knowledge, also applies the DNNs, such as ResNets and Fractional-DNNs, to ill-posed problems such as Maxwell's equations with Gauss's law. A comparison of these standard DNNs with their time-step variable versions has also been carried out. The problem is ill-posed in the following sense: 
N\'ed\'elec finite elements are traditionally used to discretize Maxwell's equations. However, they are curl-conforming and thus the Gauss's law (divergence condition) cannot be directly imposed \cite{Ciarlet2014}. The DNNs are shown to generalize well on the unseen data and physical domains.

%
%
%
Deep learning is a nascent field of research with many exciting applications, for example imaging science \cite{antil2020fractional,he2016deep,jin2017deep,rawat2017deep,wu2018deep}, biomedical applications \cite{chen2018vox,hammernik2018learning,lee2018deep}, satellite imagery, remote sensing \cite{bischke2017detection,tai2017image,zhang2018missing}, segmentation \cite{ronneberger2015u}, and gaming \cite{silver2017mastering}.
Recently, this topic is starting to receive significant attention from mathematicians
\cite{R.DeVore_BHanin_GPetrova_2021a,weinan2019machine}. Especially, the fact that DNNs of type \eqref{eq:yintro} can be viewed as optimization problems constrained by discrete dynamical systems \cite{antil2020fractional,benning2019deep,chang2018reversible,gunther2020layer,haber2017stable,lu2018beyond,schoenlieb2019research} and partial differential equations \cite{liu2020selection,ruthotto2020deep}. In these settings, the state-of-the-art is to consider an equidistant time grid, where the time step-size $\tau$ is chosen before running the optimization algorithm to identify weights, i.e. $\tau$ is a hyperparameter.

The articles \cite{hayou2021stable,JagtapAdaptive} suggest to consider $\tau^{[\ell]}$ as hyperparameters for physics-informed neural networks and ResNets, respectively. In both cases the hyperparameters $\tau^{[\ell]}$ are seen as scalings applied, outside \cite{hayou2021stable} and inside \cite{JagtapAdaptive}, the activation function. 
In \cite{JagtapAdaptive}, a global $\tau=\tau^{[\ell]}$ for all $\ell$ is considered leading to training error improvement and more accurate solutions,
which will also hold true for our more general setting. In \cite{hayou2021stable}, 
a sequence of $\tau^{[\ell]}$, fulfilling a probabilistic condition, is chosen to
avoid exploding gradients. In contrast, we let the optimization algorithm learn 
$\tau^{[\ell]}$ and provide deterministic arguments to overcome the vanishing 
and exploding gradient problems.

\smallskip
\noindent
{\bf Outline:} 
This article is organized as follows. Section \ref{sec:Prelim} introduces some basic preliminary results and notation.
We introduce definitions of fractional derivatives and state a generic DNN. We also describe an extension of 
this DNN to include a recently introduced bias ordering idea from \cite{antil2021bias} which offers multiple 
advantages such as narrowing the parameter search space.
This is followed by Section \ref{sec:contDNNs} where the relation between continuous DNNs and dynamical systems 
is considered. Special attention is given to two continuous version of DNNs: ResNet (DNN with standard time
derivative) and Fractional-DNN (DNN with fractional time derivative). Stability results for these two DNNs are also 
provided. Notice, that Fractional-DNNs have been recently introduced in \cite{antil2020fractional} for classification and further 
extended in \cite{antil2021novel} to inverse problems with PDEs. They have multiple advantages over ResNets as they can 
incorporate memory into the network due to the nonlocal nature of fractional derivatives and due to the low regularity 
requirements of fractional derivatives, they can be applied to non-smooth functions. As stated above, one of the main motivations behind introducing Fractional-DNN was to overcome vanishing and exploding gradients.  The article \cite{antil2020fractional} provides numerical evidence of overcoming the vanishing gradient problem.

Next, in Section \ref{sec:architecture} we state selected DNN architectures and compare them for a fixed $\tau$. 
The new framework with variable $\tau^{[\ell]}$ is applied to the architectures of Section \ref{sec:architecture} in 
Section~\ref{sec:Learning}. Notice that our approach is broad and can be applied to any network architecture 
of type \eqref{eq:yintro} and it is independent of the choice of the loss functional. 
Clearly, the proposed approach inherits all the positive aspects of these existing networks.
Additionally, the new framework is rigorously shown to overcome vanishing and 
exploding gradient issues (cf. Section \ref{sec:Gradients}).

Finally, in Section~\ref{sec:Num} we illustrate the efficacy of our approach with the help of an ill-posed 3D-Maxwell's equation. The numerical examples validate the above mentioned advantages of the proposed framework. In particular, stability and network reduction.

\section{Preliminaries} \label{sec:Prelim}

The goal of this section is to introduce the relevant notation and abstract optimization problems
arising while training the DNNs. The content of this section is well-known \cite{antil2020fractional,antil2021novel,antil2021bias}.

\begin{table}[h]
\begin{center}
	\begin{tabular}{l l}
		\hline
		Symbol $\qquad$ & Description \\
		\hline
		$L \in \mathbb{N}$ & Number of network layers (i.e. network depth)\\
		$N \in \mathbb{N}$ & Number of distinct data samples\\
		$n_\ell$ & Number of nodes in layer $\ell$ \\
		$y^{[\ell]} \in \mathbb{R}^{n_\ell}$ & Feature vector in layer $\ell$\\
		$\sigma$ & Activation function \\
		$W^{[\ell]} \in \mathbb{R}^{n_{\ell +1} \times n_{\ell }}$ & Weights in layer $\ell$\\
		$b^{[\ell]} \in \mathbb{R}^{n_{\ell +1}}$ & Biases in layer $\ell$\\
		$\tau^{[\ell]} \in \mathbb{R}$ & Time step-size in layer $\ell$\\
		$\theta^{[\ell]} \in \mathbb{R}^{n_{\ell +1}(n_\ell +1)+1}$ & Vector of all variables (weights, biases and time step-size) in layer $\ell$\\
		$P_j^{\ell}$ & Projection matrix from layer $j$ onto layer $\ell$\\ 
		$\phi$ & Adjoint variables \\
		$\mathcal{F}$ & Network represented as a function\\
		$f_\ell$ & Layer function \\
		$J$ & Loss function\\
		$\lambda_1,\lambda_2,\beta \in \mathbb{R}$ & Regularization parameters\\
		$\mathcal{L}$ & Lagrangian \\
		$\{u,S(u)\}$ & Input / Output pair of training data\\
		$\Gamma(\cdot)$ & Euler's Gamma function\\
		\hline
	\end{tabular}
\end{center}
\end{table}

\subsection{Caputo fractional derivative} \label{subsec:Caputo}

In preparation for the Fractional-DNN architecture, we next introduce the left and right 
Caputo fractional derivatives for absolutely continuous functions and refer to 
 \cite[Definitions 2.1 and  2.4, and Proposition 2.3]{antil2021unified} and 
 \cite[(2.4.17) and (2.4.18)]{kilbas2006theory} for details.

\begin{definition} (Left Caputo Fractional Derivative)
	Let $y \in W^{1,1}([0,T];X)$, with $X$ denoting a Banach space. The left Caputo fractional 
	derivative of order $\gamma \in (0,1)$ is given by
	\begin{equation} \label{Lcaputo}
		\partial_t^{\gamma} y(t) 
		=c_\gamma
		 \int_{0}^{t}\frac{y'(r)}{(t-r)^{\gamma}} d{r} ,
	\end{equation}
	where $c_\gamma := \frac{1}{\Gamma(1-\gamma)}$ and  $\Gamma(\cdot)$ is Euler's Gamma function.
\end{definition}
\begin{definition} (Right Caputo Fractional Derivative)
	Let $y \in W^{1,1}([0,T];X)$, with $X$ denoting a Banach space. The right Caputo fractional 
	derivative of order $\gamma \in (0,1)$ is given by
	\begin{equation*}
		\partial_{T-t}^{\gamma} y(t) = - c_\gamma 
		 \int_t^{T} \frac{y'(r)}{(r-t)^{\gamma} } dr.
	\end{equation*}
\end{definition}

Next, we introduce the general deep learning problem as an optimization problem with 
DNN constraints.

\subsection{Deep Learning problem} \label{subsec:Opt}
Consider a neural network architecture with an input layer of dimension $n_0$, $L-1$ hidden layers of dimension $n_\ell$ for $\ell=1,\ldots,L-1$ and an output layer of dimension $n_L$. Then we can represent this network as a function 
\begin{equation} \label{eq:network}
	\mathcal{F} = f_{L-1} \circ f_{L-2} \circ \cdots \circ f_0,
\end{equation}
where $\{f_\ell\}_{\ell=0}^{L-1}$ are the layer functions. These layer functions are parameterized by weight matrices $W^{[\ell]}  \in \mathbb{R}^{n_{\ell +1} \times n_{\ell }}$ and bias vectors $b^{[\ell]} \in \mathbb{R}^{n_{\ell+1}}$. The definition of $f_\ell$ depends on the network architecture.
We will introduce different options in Section \ref{sec:architecture}.

The weights and biases are identified during a training process that requires solving an optimization problem. Let  $\left\{u^{(i)}, S(u^{(i)})\right\}_{i=1}^N$ (input/ output pairs) denote the training data. Then the goal is to match the output of the DNN with the data points
$S(u^{(i)})$. This is accomplished by minimizing a loss functional $J$ and the resulting optimization problem is given by:
\begin{equation}\label{eq:orig}
\begin{aligned}
	&\min_{ \{ W^{[\ell]}\}_{\ell=0}^{L-1},  \{ b^{[\ell]} \}_{\ell=0}^{L-2}      } J \left(  \{ (y^{[L](i)}, S(u^{(i)})) \}_i,  \{ W^{[\ell]}\}_\ell , \{ b^{[\ell]} \}_\ell \right) \\
	&\text{subject to} \qquad y^{[L](i)} = \mathcal{F} \left(u^{(i)}; (  \{ W^{[\ell]}\}_\ell,  \{ b^{[\ell]} \}_\ell) \right) \quad i = 1, \ldots, N.
\end{aligned}
\end{equation}
%
One standard choice for the loss function is the mean squared error
\begin{equation*}
	J := \frac{1}{2N} \sum_{i=1}^{N} || y^{[L](i)} - S(u^{(i)})||_2^2.
\end{equation*}
Another example is the cross-entropy, see for more examples \cite{Goodfellow-et-al-2016}. The choice of $J$ is dictated by the application. For our cause, it is not relevant which of these options is chosen, since our  focus is on DNNs represented by the operator $\mathcal{F}$.

It is also common to add regularization, for example, $\ell^1$ and $\ell^2$ regularizations on weights and biases
\begin{equation}\label{eq:origJlamb}
 J_{\lambda_1} = J + \frac{\lambda_1}{2} \sum_{\ell=0}^{L-1} \left(  ||W^{[\ell]} ||_2^2 + || W^{[\ell]} ||_1 \right) 
 + \frac{\lambda_1}{2} \sum_{\ell=0}^{L-2} \left(  ||b^{[\ell]} ||_2^2 + || b^{[\ell]} ||_1 \right),
\end{equation}
where $\lambda_1 >0$ is the regularization parameter.

Before, we continue, we emphasize that recently, the article \cite{antil2021bias} has extended the above generic
network \eqref{eq:orig} by incorporating ordering among the bias vector components in each layer. The main idea is 
that in each layer $\ell$, with $\ell = 0,\dots, L-2$, one enforces 
	\begin{equation}\label{eq:bias_order}
		b^{[\ell]}_j \le b^{[\ell]}_{j+1}, \quad j=1,\ldots,n_{\ell +1}-1, 
	\end{equation} 
where the subscript $j$ indicates the bias vector component. This approach offers multiple advantages as highlighted
in \cite{antil2021bias}. Our numerical examples further provides a comparison between with and without bias ordering
framework. Notice that the bias ordering is implemented using a penalty framework, see \cite{antil2021bias} for details.

 Next, we provide a mathematical background behind learning the time step-sizes $\tau^{[\ell]}$.
 To start, we discuss a link between some DNNs and dynamical systems.

\section{Continuous DNNs} \label{sec:contDNNs}


In this section, we study the continuous structure of multiple DNNs, cf.~\eqref{eq:yintro} and \eqref{eq:network}.
This section mainly focuses on the stability of these architectures, which will follow from a connection with dynamical systems. For other approaches we refer to \cite{Avant2021ABLR,hayou2021stable} and references therein. Since we are primarily interested in stability results, we start with a basic remark on a (finite) network with a Lipschitz activation function, like ReLU. The finite composition of Lipschitz functions is also a Lipschitz function, and therefore differentiable almost everywhere by Rademacher's theorem. 
In the following we show some historical connections between neural networks and dynamical systems, and later on we consider continuous Fractional-DNNs 
which enables memory into the DNNs. 
 \subsection{Ordinary Dif{}ferential Equations and Neural Networks}
\label{subsec:ODEsandNN}
The relation between Neural Networks (NNs)  and dif{}ferential equations is not new. In fact, in the late '80s, in  \cite{Pineda87}  the following model was considered for the activity of $j$-th neuron:   
\begin{equation}\label{feedfor}
\frac{d y_j}{dt} = -\alpha y_j + \beta \sigma \Bigl( \sum_k  w_{jk} y_k\Bigr) + b_j,
\end{equation}
where $\alpha$ and $\beta$ are (given) positive constants, the weights $\{w_{jk}\}$ represent the connection strength between the $k$-th and $j$-th neurons, and $b_j$ represents a bias. Note that most modern neural network architectures are related to stationary solutions of the ODE above.  
In the present work we restrict our focus to a different connection with dynamical systems (cf.~\eqref{eq:yintro}), which is more  recent, mainly because it has been tested more thoroughly. Also, because we are interested in optimal control, we will primarily focus on the ideas presented in \cite{antil2020fractional}, \cite{neuralOdes}, and \cite{ruthotto2020deep}. Nevertheless, for completeness, we also mention some other related works \cite{Weinan17}  and \cite{Sonoda19} from a dynamical systems point of view, \cite{Tarasov2009} for universal maps with memory, \cite{lifourier21} for problems in the frequency domain,  \cite{RK93} for a Runge-Kutta based NN, PINNs \cite{PINNs21} for PDE-related problems, and SINDy \cite{SINDy18} for data-driven model discovery. Also, 
when ReLU is considered as the activation function, a DNN is a  high-dimensional, piecewise linear function. Therefore, some techniques from the Finite Element and the Monte Carlo methods can be used for its analysis, cf. \cite{Jinchao}.

Motivated by ResNets \cite{he2016deep}, the authors in \cite{neuralOdes} relates DNNs of type \eqref{eq:yintro} 
to a recurrence relation obtained when numerically solving a system of ODEs. For instance, for  given $f:\mathbb{R}\times \mathbb{R}\mapsto \mathbb{R}$ and $y_0\in \mathbb{R}$,  consider the problem: Find a function $y$, such that:
\begin{align}
\begin{aligned}
      y '(t) &=  f(t, y(t)),  \quad \mbox{ in } (0,T),\label{def:ode}\\
   y(0)&= y_0.
\end{aligned}
\end{align}
A solution for this system can be approximated, under mild assumptions on $y$ and $f$, by the Euler method:  
\begin{equation*}
     y(t+\tau)-  y(t) =\int_{t}^{t + \tau}  y'(s) ds
    = \int_{t}^{t + \tau} \!\! f( s, y(s)) ds \approx \tau  \cdot  f (t, y(t)),
\end{equation*}
where we have (formally) used the fundamental theorem of calculus and a left Riemann-sum approximation. Namely, we can understand the layers of a DNN as samples from a continuous system that evolves from the input to the output. Here, the f{}irst and last layers are special cases due to common upsampling/downsampling techniques.

It is worth mentioning, that in \cite[B.2]{neuralOdes} the authors consider $ f= f(y(t),t,\theta) $, where $\theta$ represents the parameters of the DNN. Namely, $\theta$  is independent of $t$ and therefore their ODE system is limited (essentially) to autonomous systems.
Thus, DNNs generated with the method given in \cite{neuralOdes} are smooth by construction. This property allows one to use well-known results in the theory of dynamical systems, control theory, adaptive ODE solvers, among others. An interesting application where smooth trajectories are desired is when 
 self-intersecting trajectories/surfaces are not allowed as in shape optimization (manifold surfaces), cf.  \cite{ShapeOptR,pointflow}.
It is clear that the additional smoothness also limits the usability of the model \cite{dupont2019}. Obviously, the architecture of a neural network must match its purpose, i.e. the given data and desired application case.

Besides the networks of type \eqref{def:ode}, the present work also focuses on problems where the system underneath depends on its history in a nonlocal way. The latter is most commonly found in systems with  \textit{Hysteresis} or delayed ef{}fects. Note that the derivative in \eqref{def:ode} is a local operator, this follows from its pointwise limit def{}inition.
Therefore, based on \cite{antil2021novel,antil2021unified,antil2020fractional} we consider a fractional derivative based approach. As pointed out in \cite{antil2021novel,antil2020fractional}, this serves two main purposes: it acts as a global operator (memory ef{}fect), and the order of a differential equation is allowed to be less than 1, which reduces the smoothness of the system.

\subsection{Stability of continuous Fractional-DNN}
\label{subsec:Stability}

\medskip
As mentioned before, ResNet-like architectures can be connected to a classical ODE system, and therefore we can apply the well-known theories 
to analyze the properties of the DNN \cite{neuralOdes,ruthotto2020deep}. A commonly desired property is the continuous dependence on the data. In the context of machine learning this means that input variables, which are ``close'' should produce outputs of the DNN which are also ``close''. Of course, in the context 
of real-life applications,  the notions of distance is not always known, neither is the right dimension, nor the smoothness/regularity for the system underneath. 
Following \cite{antil2021novel,antil2020fractional}, we consider a DNN architecture that can be related to a different notion of derivative, the so-called fractional derivative, see Section \ref{subsec:Caputo}, and  here we show a stability result for this notion of derivative with respect to the initial data.  In order to do so, we consider 
  $\Omega$ to be an open, bounded and connected subset of $\mathbb{R}^d,$  define
  $E:=\left(L^2(\Omega),\|\cdot\|_\Omega \right)$, and let 
$f:\mathrm{Dom}(f) \subseteq [0,\infty)\times E\mapsto E$. 
To establish the stability of continuous Fractional-DNN, we consider a dynamical system for $y$. 
Notice that similar structure holds for the continuous Fractional-DNN (cf. \eqref{eq:geneqn})
\begin{align}
\begin{aligned}
 \partial_t^\gamma y&=f(t,y), \qquad \mbox{with} \qquad 
 y(0)=y_0,
 \end{aligned}\label{def:FODEs}
\end{align}
where $y_0 \in E,$ and $f$ satisfies the standard assumptions: 
\begin{align}
\begin{cases}
\begin{aligned}
&\mbox{There exist positive constants $T$ and $r$ such that $f$ restricted to   $[0,T]\times B_r(y_0)$ }\\
&\mbox{is continuous, bounded and Lipschitz with respect to the second argument.} 
\end{aligned}
\tag{H} \label{hyp:f}
\end{cases}
\end{align}
The last hypothesis implies there exists $L>0$ such that 
\begin{align*}
\|f(t,y_1)-f(t,y_2)\|_\Omega \leq L \|y_1-y_2 \|_\Omega \quad \forall t\in [0,T], \mbox{ and } \forall y_1,y_2 \in B_r(y_0). 
\end{align*}
Let us remark that $E$ can be replaced by any other space with Radon-Nikodym property but based on the most common loss functions we restrict the analysis to $L^2(\Omega).$ From \cite{antil2021unified} we have the following result connecting the strong and generalized Caputo derivatives
\begin{lemma}[] \label{lemma:eqStongGen}
Let $\gamma\in (0,1)$, and $T>0$.  
If ${y\in W^{1,1}\left((0,T); E \right)}$ then the following equality holds in the $L^1((0,T];E)-$sense%
\begin{align}
    \partial^\gamma_t y(t) = D^\gamma_t(y-y(0))(t), \label{eq:RLeqCaputo}
\end{align}
for a.e. $t\in (0,T]$, where $D^\gamma_t$ denotes the {\it Left Riemann-Liouville fractional derivative}, cf. 
\cite[Def{}inition 2.2]{antil2021unified}.
\end{lemma}
\begin{proof}
The proof follows from \cite[Proposition 2.3]{antil2021unified}, and the fact that every ref{}lexive Banach space has the {\it Radon-Nikodym property}.
\end{proof}
We write the Left Caputo derivative in the generalized Caputo derivative form \eqref{eq:RLeqCaputo},
because several of the well-known results, which hold for standard ODEs, also have their counterparts 
in the generalized Caputo derivative setting.
 For instance, the solution operator for a non-autonomous fractional ODE 
 can be represented in terms of a Volterra integral, cf.~\cite[Theorem 2.1]{Diethelm2002}. By using this 
 integral representation,  the  next proposition shows the stability of the fractional ODE 
 \eqref{def:FODEs} with respect to its initial value, when the solution is smooth enough.
\begin{proposition} 
Given  $y_0\in E$, and  $f$ that satisf{}ies (\ref{hyp:f}). If $y_\alpha, y_\beta \in W^{1,1}((0,T); E)$ solve (\ref{def:FODEs}) with initial conditions 
$y_{\alpha,0}$ and $y_{\beta,0}$  both in  $B_r(y_0)$. Then,   
\begin{align}
\|y_\alpha -y_\beta \|_{L^1(0,T;E)}\leq C \|y_{\alpha,0}-y_{\beta,0} \|_\Omega,
\end{align} 
where $C=C(\gamma,T,L)>0$.  
\end{proposition}
\begin{proof}  By Lemma \ref{lemma:eqStongGen}, and because $E$ is ref{}lexive and therefore has the Radon-Nikodym property, we can recast (\ref{def:FODEs}) as  (\ref{eq:RLeqCaputo}), with $y(0)\in \{y_{\alpha,0},\, y_{\beta,0}\}$, and represent each solution  in terms of a  nonlinear Volterra integral, cf. \cite[Lemma 2.1]{Diethelm2002}. Namely, if 
$y_\alpha, y_\beta$ represent the solutions for (\ref{def:FODEs}) with corresponding initial conditions 
$y_{\alpha,0}, y_{\beta,0}$, then for $t\in [0,T]$, and a.e. $x\in \Omega$
\begin{align*}
    y_\alpha(x,t)= y_{\alpha,0}(x)+ c_\gamma \int_{0}^t \frac{1}{(t-\tau)^{\gamma}}f(\tau,y_\alpha(x,\tau))d\tau,\\
    y_\beta(x,t)= y_{\beta,0}(x)+ c_\gamma\int_{0}^t \frac{1}{(t-\tau)^{\gamma}}f(\tau,y_\beta(x,\tau))d\tau ,
\end{align*}
where we recall $c_\gamma = \frac{1}{\Gamma(1-\gamma)}$.
Then,
\begin{align*}
 \|y_\alpha(\cdot,t)-y_\beta(\cdot,t)\|_\Omega 
  &\leq   \left\|y_{\alpha,0} -y_{\beta,0} \right\|_\Omega + c_\gamma\int_{0}^t \frac{1}{(t-\tau)^{\gamma}}\left\| f(\tau,y_\alpha)-f(\tau,y_\beta)  \right\|_\Omega d\tau \\
   &\leq   \left\|y_{\alpha,0} -y_{\beta,0} \right\|_\Omega + c_\gamma\int_{0}^t \frac{1}{(t-\tau)^{\gamma}}
   L\left\| y_\alpha(\cdot,\tau)-y_\beta(\cdot,\tau)  \right\|_\Omega d\tau
\end{align*}
Finally, from Gronwall's inequality in its integral form  
\begin{align*}
\|y_\alpha(\cdot,t)-y_\beta(\cdot,t)\|_\Omega&\leq  \|y_{\alpha,}-y_{\beta,0}\|_\Omega \exp \left( 
 \frac{L}{\Gamma(1-\gamma)} 
 \int_{0}^t \frac{1}{(t-\tau)^{\gamma}} d\tau \right)\\
&= \|y_{\alpha,0}-y_{\beta,0}\|_{\Omega} \exp \left(
 \frac{L}{\Gamma(1-\gamma)} 
 \frac{t^{1-\gamma}}{1-\gamma} \right),
\end{align*} 
and integrating over $(0,T)$ concludes the proof.
\end{proof}
%
\begin{remark}
It is important to point out that the previous results assume the regularity $W^{1,1}((0,T);E)$, but for most problems in Machine Learning the regularity of solutions is still an open question. Even at the ``discrete level'' the regularity depends on the data, DNN architecture, optimization algorithm, loss function, among others factors.
Another difficulty is that DNNs can have different number of neurons in each layer, i.e., the space $E$ can change in time.
\end{remark}

Finally, and as mentioned before, a DNN with Lipschitz activation functions defines a locally Lipschitzian operator.
 Later in Section~\ref{sec:Gradients}, we will explore how the activation function and weights affect locally the gradient of a DNN, and therefore the Lipschitz constant,
and we will study the vanishing and exploding gradients problem of various DNNs under the variable$-\tau$ framework which is introduced  in Section~\ref{sec:Learning}.


\section{Network architectures with fixed $\tau$-parameter} \label{sec:architecture}
Let us begin by stressing that in general the proposed framework with variable $\tau$ can be applied to any DNN.
We will illustrate our ideas using three representative DNNs.
Subsequently, we will describe their strengths and weaknesses.

The first network architecture is ResNet \cite{he2016deep}. 
As described in the previous section (see \eqref{def:ode}), this network arises after adding an identity 
map to a standard feedforward network. This leads to connectivity between the adjacent 
layers. In order to connect all layers and additionally be able to approximate non-smooth 
functions, we refer to DenseNet \cite{huang2017densely} and Fractional-DNN 
\cite{antil2020fractional}.
We remark that there also exist other approaches that attempt to induce multilayer connections, e.g. Highway Net \cite{HighwayNet2015}, AdaNet \cite{AdaNet2017}, ResNetPlus \cite{ResNetPlus2018}, etc.

DenseNet is an ad-hoc method that uses the feature maps of all preceding layers as inputs into all subsequent layers. 
Meanwhile, Fractional-DNN can be viewed as a time-discretization of a fractional in time non-linear ODE of type \eqref{def:FODEs}, connecting all layers in a mathematically rigorous manner. Both approaches, DenseNet and 
Fractional-DNN, improve the vanishing gradient effect issue due to the memory effect incorporated. Furthermore, 
Fractional-DNN allows approximation of non-smooth functions and thus can also potentially help with exploding 
gradients.

We recall the \textbf{ResNet} \cite{he2016deep} with equidistant time-steps $\tau$, cf. \eqref{eq:yintro}. The feature vector 
$y^{[\ell]} \in \mathbb{R}^{n_\ell}$ in layer $\ell=1,\ldots,L$ is computed by forward propagation  in the 
following way
\begin{align*}
	y^{[\ell]} &= P_{\ell-1}^{\ell} y^{[\ell-1]} + \tau \sigma \left(  W^{[\ell-1]} y^{[\ell-1]} + b^{[\ell-1]} \right), \qquad \ell=1,\ldots,L,
\end{align*}	
where $P_0^1=0 \in \mathbb{R}^{n_0 \times n_1}$ and $y^{[0]} = u$.
Here, $\sigma$ is a nonlinear activation function, for instance, ReLU \cite{Goodfellow-et-al-2016}, 
$\tau$ is the fixed time-step length and $u$ is the input data. Notice that, if all the layers are of same size, then $P_{\ell-1}^{\ell}$ equals an identity matrix. In general, $P_{\ell-1}^{\ell}$ will allow layers to have different sizes, i.e.,
\begin{equation*}
	\dim (P_{\ell-1}^\ell y^{[\ell-1]}) = \dim (y^{[\ell]}).
\end{equation*}
While ResNet does introduce connectivity between adjacent layers, we are also interested in fully connected networks, for instance, a DenseNet \cite{huang2017densely}.

In a \textbf{DenseNet} the connection through all layers is achieved by the following forward propagation 
\begin{align*}
	y^{[\ell]} &= \sum_{i=0}^{\ell-1} P_{i}^\ell y^{[i]} + \sigma \left(  W^{[\ell-1]} y^{[\ell-1]} + b^{[\ell-1]} \right), \qquad \ell=1,\ldots,L,
\end{align*}
where $y^{[0]} = u$ is the input data.

Notice that this method does not contain a time step-size like parameter. It is possible to artificially add a parameter $\tau$ before the activation function $\sigma$. The connection of the resulting expression to a dynamical system remains unclear.
Instead of DenseNet, we focus on Fractional-DNN. In addition, to connecting all layers, the Fractional-DNN can be understood as a time-discretization of a dynamical system of type \eqref{def:FODEs}. Hence, learning the time step-sizes is a meaningful task in this setup.

We recall the \textbf{Fractional-DNN}, with equidistant time-steps $\tau$, from \cite{antil2020fractional,antil2021novel}. 
It corresponds to a time-discretization of a system of type \eqref{def:FODEs}. 
The forward propagation for the Fractional-DNN is given by
\begin{align*}
	y^{[\ell]} &= P_{\ell-1}^\ell y^{[\ell-1]} - \sum_{j=1}^{\ell-1} a_{\ell-j} (P_{j}^\ell y^{[j]} - P_{j-1}^\ell y^{[j-1]}) + \tau^{\gamma} \Gamma(2-\gamma) \sigma   \left(  W^{[\ell-1]} y^{[\ell-1]} + b^{[\ell-1]} \right),
\end{align*}
where $\ell=1,\ldots,L$, $y^{[0]} = u$, and $P_j^\ell$ as before. Moreover 
\begin{equation*}
	a_{\ell-j}:= (\ell+1-j)^{1-\gamma} - (\ell-j)^{1-\gamma}.
\end{equation*}
\begin{remark}
In Section \ref{subsec:learningfrac} below, we will consider a Fractional-DNN 
with variable $\tau$, i.e., $\tau^{[\ell]}$ for $\ell=1,\ldots,L-1$. In this case, the 
coefficients $a_{\ell-j}$ will depend on $\tau^{[j]},\ldots,\tau^{[\ell]}$.
\end{remark}
We conclude this section by emphasizing that depending on the number of DNN 
outputs, the last layer may have different size, which can be captured via
	\begin{equation*}
		y^{[L]} = W^{[L-1]} y^{[L-1]}.
	\end{equation*}
For the remainder of the paper, we will assume such a setup for the last layer.

\section{Variable-$\tau$ framework for DNNs} \label{sec:Learning}
%
Instead of a fixed $\tau$, we propose to use a different $\tau^{[\ell]}$ for each
layer, which is learned during the training process. This allows us to optimize the 
``time step-sizes" $\tau^{[\ell]}$, resulting in what can be viewed as an adaptive 
time-discretization of the ODE tailored to the optimization (learning) problem.
The resulting optimization problem is given by (cf.~\ref{eq:orig})
\begin{equation}\label{eq:Ptau}
\begin{aligned}
	&\min_{ \{ W^{[\ell]}\}_{\ell=0}^{L-1},  \{ b^{[\ell]} \}_{\ell=0}^{L-2}  , \{ \tau^{[\ell]} \}_{\ell=0}^{L-2}    } J_{\lambda} \left(  \{ (y^{[L](i)}, S(u^{(i)})) \}_i,  \{ W^{[\ell]}\}_\ell , \{ b^{[\ell]} \}_\ell , \{ \tau^{[\ell]} \}_\ell \right) \\
	&\text{subject to} \qquad y^{[L](i)} = \mathcal{F} \left(u^{(i)}; (  \{ W^{[\ell]}\}_\ell,  \{ b^{[\ell]} \}_\ell, \{ \tau^{[\ell]} \}_\ell) \right) \quad i = 1, \ldots, N.
\end{aligned}
\end{equation}
Constraints on $\tau^{[\ell]}$, for instance, non-negativity can be easily incorporated.
Recall, from \eqref{eq:origJlamb} that $J_{\lambda_1}$ contains the regularization for 
the weights $W^{[\ell]}$ and $b^{[\ell]}$. Additional regularization on $\tau^{[\ell]}$ can 
be easily introduced as
\begin{equation*}
	J_{\lambda} := J_{\lambda_1} + \frac{\lambda_2}{2} \sum_{\ell =0}^{L-2} \left( || \tau^{[\ell]} ||^2_2 + || \tau^{[\ell]} ||_1  \right),
\end{equation*}
where $\lambda = \lambda(\lambda_1,\lambda_2)$. 

We apply the $\tau$-variable framework to the ResNet and the Fractional-DNN discussed above.

\subsection{ResNet with variable $\tau$}

Consider \eqref{eq:Ptau} with $\mathcal{F}$ denoting the ResNet with variable $\tau$
\begin{equation}\label{eq:OVarRes}
	\begin{aligned} 
		y^{[\ell]} &= P_{\ell-1}^\ell y^{[\ell-1]} + \tau^{[\ell-1]} \sigma (W^{[\ell-1]} y^{[\ell-1]} + b^{[\ell-1]}), \qquad \ell=1,\ldots, L-1, \\
		y^{[L]} &= W^{[L-1]} y^{[L-1]}. 
	\end{aligned}
\end{equation}
For simplicity of notation, we write $J$ instead of $J_\lambda$ and collect $W^{[\ell]}$, $b^{[\ell]}$, 
and $\tau^{[\ell]}$ for all $\ell$ into one vector $\theta$ and denote the adjoint variables by $\phi$.

Following the approach from \cite{antil2020fractional} and to derive the optimality system, 
we introduce the Lagrangian functional
\begin{align*}
	\mathcal{L} (y, \theta, \phi) = &J (\theta) 
	- \sum_{\ell=1}^{L-1} \left\langle y^{[\ell]} - P_{\ell-1}^\ell y^{[\ell-1]} - \tau^{[\ell-1]} \sigma
	(W^{[\ell-1]} y^{[\ell-1]} + b^{[\ell-1]}), \phi^{[\ell]} \right\rangle \\ 
	&- \left\langle y^{[L]} - W^{[L-1]} y^{[L-1]} , \phi^{[L]} \right\rangle .
\end{align*}
Setting the variation of $\mathcal{L}$ with respect to $\phi$ equals zero, we recover  the
\emph{state equation} \eqref{eq:OVarRes}.
Similarly, setting the variation of $\mathcal{L}$ with respect to $y$ equals zero, we 
arrive at the \emph{adjoint system}
\begin{align*}
	\phi^{[\ell]} &= (P_{\ell}^{\ell+1})^{\top}\phi^{[\ell+1]} 
	+ \tau^{[\ell]} (W^{[\ell]})^{\top} \left( \phi^{[\ell+1]} \odot  \sigma' ( W^{[\ell]} y^{[\ell]} + b^{[\ell]} )  \right) , \qquad \ell=L-2, \ldots, 1, \\
	\phi^{[L-1]} &=  (W^{[L-1]})^{\top} \phi^{[L]} , \\
	\phi^{[L]} &= \partial_{y^{[L]}} J (\theta) = y^{[L]} - S(u),
\end{align*}
where the last equality is due to the specific choice of the least-squares loss function. 
It will be different in case of the cross-entropy softmax, for instance.

Since we will be solving the above optimization problem using a gradient-based method, 
we also need to evaluate the derivatives with respect to $\theta$:
\begin{align*}
	\partial_{W^{[L-1]}} \mathcal{L} &= \phi^{[L]} (y^{[L-1]})^{\top} + \partial_{W^{[L-1]}} J (\theta), \\
	\partial_{W^{[\ell]}} \mathcal{L}     &= y^{[\ell]} \left( \phi^{[\ell+1]} \odot \tau^{[\ell]} \sigma' (W^{[\ell]} y^{[\ell]} +b^{[\ell]}) \right)^{\top} + \partial_{W^{[\ell]}} J(\theta),  && \ell=0,\ldots,L-2, \\
	\partial_{b^{[\ell]}} \mathcal{L}      &=  (\phi^{[\ell+1]})^{\top} \tau^{[\ell]} \sigma' \left( W^{[\ell]} y^{[\ell]} + b^{[\ell]} \right) + \partial_{b^{[\ell]}} J(\theta),   && \ell=0,\ldots,L-2, \\
	\partial_{\tau^{[\ell]}} \mathcal{L}  &= \left\langle  \sigma (W^{[\ell]} y^{[\ell]} + b^{[\ell]}), \phi^{[\ell+1]} \right\rangle + \partial_{\tau^{[\ell]}} J(\theta), && \ell=0,\ldots,L-2 . 
\end{align*}

Next, we state the Fractional-DNN \cite{antil2020fractional} but now with variable $\tau^{[\ell]}$. Recall that, in contrast to a ResNet, 
the Fractional-DNN allows connectivity between all the layers.

\subsection{Fractional-DNN with variable $\tau$} \label{subsec:learningfrac}

Consider a time-discretization $t_0 \le t_1 \le \cdots \le t_L$ with $L \in \mathbb{N}$ and set
$I_\ell := (t_\ell,t_{\ell+1}]$ and $\tau^{[\ell]} = t_{\ell+1} - t_\ell$ for $0 \le \ell \le L-1$. 
Throughout, we will assume that $\tau^{[\ell]} > 0$ to justify division by $\tau^{[\ell]}$.

We generalize the numerical scheme introduced in \cite{YLin_CXu_2007a,YLin_XLi_CXu_2011a}
to a non-equidistant time discretization and obtain the discrete approximation of the left-sided Caputo 
fractional derivative of order $\gamma \in (0,1)$. For $0 \le \ell \le L-1$, we have that:
\begin{align}
	\partial_t^{\gamma} y(x,t_{\ell+1}) &= c_\gamma \int_0^{t_{\ell+1}} \frac{\partial_t y(x,t)}{(t_{\ell+1}-t)^{\gamma} } dt
	= c_\gamma  \sum_{j=0}^\ell \int_{I_j} \frac{\partial_t y(x,t)}{(t_{\ell+1}-t)^{\gamma} } dt   \nonumber	\\ 
	&= c_\gamma  \sum_{j=0}^\ell \frac{y(x,t_{j+1}) -y(x,t_j)
	}{\tau^{[j]}} \int_{I_j} \frac{1}{(t_{\ell+1}-t)^{\gamma}} dt + r_{\gamma}^{\ell+1}
	\label{eq:fdiff}	
\end{align}
where we have used the finite difference approximation. Here $r_{\gamma}^{\ell+1}$ 
denotes the remainder from the Taylor formula which can be estimated as described in 
\cite[Section 3.2.1]{RHNochetto_EOtarola_AJSalgado_2014b}. After carrying 
out the integration in \eqref{eq:fdiff}, we arrive at
\begin{align}
	\partial_t^{\gamma} y(x,t_{\ell+1}) 	
	&= \textstyle c_\gamma  \sum_{j=0}^\ell \frac{y(x,t_{j+1}) -y(x,t_j)}{\tau^{[j]}} \frac{1}{(1-\gamma)} \left( \left( \sum_{i=j}^\ell \tau^{[i]} \right)^{1-\gamma} - \left( \sum_{i=j+1}^\ell \tau^{[i]} \right)^{1-\gamma} \right) + r_{\gamma}^{\ell+1}   \notag \\
	&= \textstyle c_{\gamma-1} 
	\sum_{j=0}^\ell \frac{1}{\tau^{[j]}} \left( \left( \sum_{i=j}^\ell \tau^{[i]} \right)^{1-\gamma} - \left( \sum_{i=j+1}^\ell \tau^{[i]} \right)^{1-\gamma} \right) \left( y(x,t_{j+1}) -y(x,t_j) \right) + r_{\gamma}^{\ell+1} .
	\label{eq:Lcapapprox}
\end{align}
Analogously, we obtain the approximation of the right-sided Caputo fractional derivative of order 
$\gamma \in (0,1)$ for $0 \le \ell \le L-1$:
\begin{align}
	\textstyle \partial_{T-t}^{\gamma}y(x,t_\ell) = - c_{\gamma -1} 
	\sum_{j=\ell}^{L-1} \frac{1}{\tau^{[j]}} \left( \left( \sum_{i=\ell}^{j} \tau^{[i]} \right)^{1-\gamma} - \left( \sum_{i=\ell}^{j-1} \tau^{[i]} \right)^{1-\gamma} \right) \left( y(x,t_{j+1}) -y(x,t_j) \right) + r_{\gamma}^{\ell} .
	\label{eq:Rcapapprox}
\end{align}
Before, we apply the above discretization to the Fractional-DNN formulation, we consider two generic nonlinear
ODEs of type \eqref{def:FODEs} (cf. e.g. \cite[Section 4.1]{antil2020fractional}) with $\gamma\in (0,1)$: 
\begin{equation}\label{eq:geneqn}
\begin{aligned}
	\partial_t^{\gamma} y(x,t) &= f(t,y(x,t)), \qquad y(x,0)=y_0,\\ 
	\partial_{T-t}^\gamma y(x,t) &= f(t,y(x,t)), \qquad y(x,T) =y_T.
\end{aligned}
\end{equation}
This also links back to Subsection \ref{subsec:Stability}, where the stability of the above continuous DNN 
was discussed. Here, we will move on to formulate the discrete version. 

Using the discretizations from \eqref{eq:Lcapapprox} and \eqref{eq:Rcapapprox} in \eqref{eq:geneqn}, 
for $0 \le \ell \le L-1$, we arrive at 
\begin{align*}
	y(x,t_{\ell+1}) &= y(x,t_\ell) - \sum_{j=0}^{\ell-1} a_{\ell,j} (y(x,t_{j+1}) - y(x,t_j)) + (\tau^{[\ell]})^{\gamma} c_{\gamma-1}^{-1} 
	f(t_\ell,y(x,t_\ell)), \\
	y(x,t_\ell) &= y(x,t_{\ell+1}) + \sum_{j=\ell +1}^{L-1} b_{j,\ell} (y(x,t_{j+1}) - y(t_j)) + (\tau^{[\ell]})^{\gamma} 
	c_{\gamma-1}^{-1} 
	f(t_\ell,y(x,t_\ell)) ,
\end{align*}
with
\begin{align*}
	a_{\ell,j}&:= \textstyle \frac{(\tau^{[\ell]})^{\gamma}}{\tau^{[j]}} \left( \left( \sum_{i=j}^\ell \tau^{[i]} \right)^{1-\gamma} - \left( \sum_{i=j+1}^\ell \tau^{[i]} \right)^{1-\gamma} \right),\\
	b_{j,\ell}&:= \textstyle \frac{(\tau^{[\ell]})^{\gamma}}{\tau^{[j]}}  \left( \left( \sum_{i=\ell}^{j} \tau^{[i]} \right)^{1-\gamma} - \left( \sum_{i=\ell}^{j-1} \tau^{[i]} \right)^{1-\gamma} \right).
\end{align*}
Notice that the equidistant case, $\tau^{[\ell]}= \tau$ for all $\ell$, considered throughout the literature, 
is a special case of the above setting. Additionally, in the equidistant setting, we have $a_{j,\ell} = b_{j,\ell}$,
which may not hold in the above generic setting.

After these preparations, we are ready to apply the $\tau$-variable framework to Fractional-DNN.
Here, we take into account that the feature vectors $y^{[\ell]}$ may have different sizes across the layers.
Thus, as in case of $\tau$-variable ResNet, we introduce projection matrices $P_{j}^\ell$ for 
$j=0,\ldots,\ell-1$ and $\ell = 1,\ldots, L-1$ with $\dim (P_{j}^\ell y^{[j]}) = \dim (y^{[\ell]})$. The 
resulting Fractional DNN with variable $\tau$ is
\begin{equation}\label{eq:state_var_tau}
\begin{aligned}
	y^{[\ell]} &= P_{\ell -1}^{\ell} y^{[\ell-1]} - \sum_{j=0}^{\ell-2} a_{\ell-1,j} (P_{j+1}^{\ell} y^{[j+1]}-P_{j}^{\ell} y^{[j]}) \\
	&\quad + (\tau^{[\ell-1]})^\gamma 
	c_{\gamma-1}^{-1} 
	\sigma ( W^{[\ell-1]} y^{[\ell-1]} + b^{[\ell-1]} ) , \qquad \ell=1,\ldots,L-1 \\
	y^{[L]} &= W^{[L-1]} y^{[L-1]}.
\end{aligned}
\end{equation}

	\begin{remark}\label{rem:scale}
		Before we proceed further, we stress that the $\tau$-variable framework is 
		not merely a scaling of the activation by $\tau^{[\ell]}$.
		Indeed, in \eqref{eq:state_var_tau} the scaling in front of $\sigma$ is not simply 
		$\tau^{[\ell]}$ but is $(\tau^{{\ell-1}})^\gamma c_{\gamma-1}^{-1}.$ 
		Furthermore, $a_{\ell-1,j}$ also contains $\tau^{[j]},\ldots,\tau^{[\ell-1]}$, which 
		makes the impact of the time step-sizes much more complex than scaling of 
		$\sigma$.
	\end{remark}
As in the ResNet case we next derive the optimality conditions. This requires introducing 
the Lagrangian formulation as before. In this fractional derivative setting, we observe a subtle issue.
It is well-known that there are two approaches to derive the optimality conditions -- optimize-then-discretize
and discretize-then-optimize \cite{HAntil_DPKouri_MDLacasse_DRidzal_2018a,MHinze_RPinnau_MUlbrich_SUlbrich_2009a}. 
Below, in the $\tau$-variable fractional setting, we observe that the two 
approaches do not coincide. It is not difficult to see that in the first case, 
optimize-then-discretize, we obtain the following adjoint equation
\begin{equation}\label{eq:adj_var_tau0}
\begin{aligned}
	\phi^{[\ell]} &= (P_{\ell}^{\ell+1})^{\top} \phi^{[\ell + 1]} + \sum_{j=\ell+1}^{L-2}b_{j,\ell} ((P_{\ell}^{j+1})^{\top}\phi^{[j+1]} - (P_{\ell}^{j})^{\top}\phi^{[j]} )  \\ 
	&\qquad + (\tau^{[\ell]})^\gamma 
	c_{\gamma-1}^{-1} 
	\left[ (W^{[\ell]})^{\top} \left( \phi^{[\ell + 1]} \odot \sigma'(W^{[\ell]} y^{[\ell]} + b^{[\ell]}) \right) \right], \qquad \ell=L-2,\ldots,1,\\
	\phi^{[L-1]} &= \left(W^{[L-1]}\right)^{\top} \phi^{[L]},\\
	\phi^{[L]} &= \partial_{y^{[L]}} J(\theta) .
\end{aligned}
\end{equation}
Next, we derive the adjoint equations for the second approach, i.e., discretize-then-optimize.
We begin by introducing the Lagrangian
\begin{align*}
	\mathcal{L} (y, \theta, \phi) = &J (\theta)
	- \sum_{\ell=1}^{L-1} \langle y^{[\ell]} - P_{\ell-1}^{\ell}y^{[\ell-1]} + \sum_{j=0}^{\ell-2} a_{\ell-1,j} (P_{j+1}^{\ell}y^{[j+1]}- P_{j}^{\ell}y^{[j]}) \\
	&\qquad \quad  - (\tau^{[\ell-1]})^\gamma 
	c_{\gamma-1}^{-1} 
	\sigma ( W^{[\ell-1]} y^{[\ell-1]} + b^{[\ell-1]} )  , \phi^{[\ell]} \rangle \\ 
	&- \left\langle y^{[L]} - W^{[L-1]} y^{[L-1]} , \phi^{[L]} \right\rangle .
\end{align*}
Setting the variation of $\mathcal{L}$ with respect to $\phi$ equal zero, we obtain the state 
equation \eqref{eq:state_var_tau}. To derive the adjoint equation, we calculate the variation of 
$\mathcal{L}$ with respect to $y^{[\ell]}$ for every $\ell=1,\ldots,L$. A detailed calculation can be found in Appendix \ref{appA1}.
Setting this variation equal to zero, we arrive at the following adjoint system
\begin{equation}\label{eq:adj_var_tau1}
\begin{aligned}
	\phi^{[\ell]} &= (1- a_{\ell,\ell-1} )(P_{\ell}^{\ell+1})^{\top} \phi^{[\ell + 1]} + \sum_{j=\ell+2}^{L-1}  (a_{j-1,\ell}-a_{j-1,\ell-1} ) (P_{\ell}^{j})^{\top} \phi^{[j]}   \\ 
	&\qquad + (\tau^{[\ell]})^\gamma 
	c_{\gamma-1}^{-1} 
	\left[ (W^{[\ell]})^{\top} \left( \phi^{[\ell + 1]} \odot \sigma'(W^{[\ell]} y^{[\ell]} + b^{[\ell]}) \right) \right], \qquad \ell=L-2,\ldots,1,\\
	\phi^{[L-1]} &= \left(W^{[L-1]}\right)^{\top} \phi^{[L]},\\
	\phi^{[L]} &= \partial_{y^{[L]}} J(\theta).
\end{aligned}
\end{equation}
Below, we collect all summands that contain factors $b_{j,\ell}$ in \eqref{eq:adj_var_tau0} on the left side, and all summands that contain factors $a_{j,\ell}$ in \eqref{eq:adj_var_tau1} on the right side. We see that the two adjoint equations given in \eqref{eq:adj_var_tau0} and \eqref{eq:adj_var_tau1} differ 
in the following term
\begin{equation*}
	\sum_{j=\ell+1}^{L-2}b_{j,\ell} ((P_{\ell}^{j+1})^{\top}\phi^{[j+1]} - (P_{\ell}^{j})^{\top}\phi^{[j]} )
	 \neq 
	 \sum_{j=\ell+1}^{L-2}a_{j,\ell} {(P_{\ell}^{j+1})^{\top}}\phi^{[j+1]} - \sum_{j=\ell+1}^{L-1}a_{j-1,\ell-1} {(P_{\ell}^{j})^{\top}}\phi^{[j]} 
\end{equation*}
In our computations, we have implemented the discretize-then-optimize approach, i.e. \eqref{eq:adj_var_tau1}.
Finally, we compute the derivative with respect to $\theta$:
\begin{align*}
	\partial_{W^{[L-1]}} \mathcal{L} &= \phi^{[L]} (y^{[L-1]})^{\top} + \partial_{W^{[L-1]}} J (\theta), \\
	\partial_{W^{[\ell]}} \mathcal{L}     &= y^{[\ell]} \left( \phi^{[\ell+1]} \odot (\tau^{[\ell]})^{\gamma} 
	c_{\gamma-1}^{-1}
	 \sigma' (W^{[\ell]} y^{[\ell]} +b^{[\ell]}) \right)^{\top} + \partial_{W^{[\ell]}} J(\theta),  && \ell=0,\ldots,L-2, \\
	\partial_{b^{[\ell]}} \mathcal{L}      &=  (\phi^{[\ell+1]})^{\top} (\tau^{[\ell]})^\gamma 
	c_{\gamma-1}^{-1}
	\sigma' \left( W^{[\ell]} y^{[\ell]} + b^{[\ell]} \right) + \partial_{b^{[\ell]}} J(\theta),   && \ell=0,\ldots,L-2, \\
	\partial_{\tau^{[\ell]}} \mathcal{L}  &=  -\sum_{k= \ell}^{L-2}  \sum_{j=0 }^{\min\{k-1,\ell\}} \partial_{\tau^{[\ell]}} (a_{k,j}) \left\langle {P_{j+1}^{k+1}}y^{[j+1]} - {P_{j}^{k+1}}y^{[j]}, \phi^{[k+1]} \right\rangle  \\ 
	&\quad+ \left\langle  \gamma (\tau^{[\ell]})^{\gamma-1}
	c_{\gamma-1}^{-1} 
	  \sigma (W^{[\ell]} y^{[\ell]} + b^{[\ell]}), \phi^{[\ell+1]} \right\rangle + \partial_{\tau^{[\ell]}} J(\theta), && \ell=0,\ldots,L-2 .
\end{align*}
Details on the computation of $	\partial_{\tau^{[\ell]}} \mathcal{L} $ can be found in Appendix \ref{appA2}.
Next, we examine the impact of variable $\tau$ onto the stability of networks such as ResNets and 
Fractional-DNNs.

\section{Vanishing and exploding gradients}
\label{sec:Gradients}

It is well known that optimization problems with DNN constraints can suffer from vanishing and exploding 
gradients, see e.g. \cite{bengio1994learning,glorot2010}. In this section, we analyze the structure of the derivatives
for several network architectures such as feedforward network, ResNet, DenseNet, Fractional-DNN, 
and the consequences of application of $\tau$-variable framework on these networks. We will identify 
various conditions to help overcome the aforementioned challenges.

For simplicity of the notation, we define the abbreviation $a^{[\ell]} :=\sigma(W^{[\ell]} y^{[\ell]} + b^{[\ell]})$ 
and omit the projection matrices $P_{\ell-1}^{\ell}$, i.e. $n_\ell = n$ for all layers $\ell$.
While the following result may not be new for the standard case with $\tau^{[\ell]} = \tau \in \R$ for all $\ell$, but to the best of our knowledge, this is new for variable $\tau^{[\ell]}$.
\begin{theorem}[Feedforward Network and ResNet]
	Consider the feedforward network and ResNet with $\tau$-variable framework
	\begin{equation*}
		\begin{aligned}
			y^{[\ell]} &= \tau^{[\ell -1]} a^{[\ell - 1]}, \qquad &&\ell = 1,\ldots,L-1,\\		
			y^{[\ell]} &= y^{[\ell-1]} + \tau^{[\ell-1]} a^{[\ell-1]},  \qquad &&\ell = 1,\ldots,L-1.
		\end{aligned}	
	\end{equation*}  
	Let $\theta^{[k]} = (W^{[k]}(:), b^{[k]}, \tau^{[k]})^{\top}$ be the parameters associated with layer $k$ for $k=0,\ldots,L-2$.  Then the respective derivatives take the form
	\begin{align}
		{\rm d}_{\theta^{[j]}}y^{[\ell]} &=  \prod_{i=\ell -1}^{j+1} \left( \tau^{[i]} {\rm d}_{y^{[i]}} a^{[i]} \right) \partial_{\theta^{[j]}} ( \tau^{[j]}  a^{[j]}), \label{eq:thetadiffvanilla} \\
		{\rm d}_{\theta^{[j]}}y^{[\ell]} &=  \prod_{i=\ell -1}^{j+1 }\left( \mathbb{I}+ \tau^{[i]}  {\rm d}_{y^{[i]}} a^{[i]}  \right) \partial_{\theta^{[j]}} (\tau^{[j]} a^{[j]} ) , \label{eq:thetadiffResNet}
	\end{align}
	for all $\ell = 1, \ldots,L-1$ and $j=0,\ldots,\ell-1$.
\end{theorem}

\begin{proof}
	For the feedforward neural network we can compute with chain rule
	\begin{align*}
		{\rm d}_{\theta^{[j]}}y^{[\ell]} &= {\rm d}_{\theta^{[j]}} \tau^{[\ell -1]} a^{[\ell-1]} \\
		&= \tau^{[\ell -1]} {\rm d}_{y^{[\ell-1]}} a^{[\ell-1]} \cdot {\rm d}_{\theta^{[j]}} y^{[\ell-1]} \\
		&= \tau^{[\ell -1]} {\rm d}_{y^{[\ell-1]}} a^{[\ell-1]} \cdot {\rm d}_{\theta^{[j]}} \tau^{[\ell -2]} a^{[\ell-2]}.
	\end{align*}
	where $\cdot$ denotes the standard matrix multiplication.
	By iterating we arrive at
	\begin{equation*}
		{\rm d}_{\theta^{[j]}}y^{[\ell]} = \prod_{i=\ell-1}^{j+1}\left(  \tau^{[i]}{\rm d}_{y^{[i]}}a^{[i]} \right) {\rm d}_{\theta^{[j]}} y^{[j+1]} =\prod_{i=\ell-1}^{j+1} \left(  \tau^{[i]} {\rm d}_{y^{[i]}}a^{[i]} \right)  \partial_{\theta^{[j]}} (\tau^{[j]}  a^{[j]}) .
	\end{equation*}
	Similarly, for the ResNet, we obtain
	\begin{align*}
		{\rm d}_{\theta^{[j]}}y^{[\ell]} &= {\rm d}_{\theta^{[j]}} ( y^{[\ell-1]} + \tau^{[\ell-1]} a^{[\ell - 1]}) \\
		&={\rm d}_{\theta^{[j]}} y^{[\ell-1]} +  \tau^{[\ell-1]} {\rm d}_{y^{[\ell-1]}} a^{[\ell - 1]} \cdot {\rm d}_{\theta^{[j]}} y^{[\ell-1]}  \\
		&= (\mathbb{I}+ \tau^{[\ell-1]} {\rm d}_{y^{[\ell-1]}} a^{[\ell - 1]} ) \,  {\rm d}_{\theta^{[j]}} y^{[\ell-1]},
	\end{align*}
where $\mathbb{I} \in \mathbb{R}^{n \times n}$ denotes the identity matrix, with $n = n_\ell$ constant throughout all layers $\ell$. 
	By iterating we arrive at 
	\begin{equation*}
		{\rm d}_{\theta^{[j]}}y^{[\ell]} =  \prod_{i=\ell-1}^{j+1 }\left( \mathbb{I}+ \tau^{[i]} {\rm d}_{y^{[i]}} a^{[i]}  \right) {\rm d}_{\theta^{[j]}} y^{[j+1]}  =  \prod_{i=\ell-1}^{j+1 } \left(\mathbb{I}+ \tau^{[i]} {\rm d}_{y^{[i]}} a^{[i]}    \right) \partial_{\theta^{[j]}} (\tau^{[j]} a^{[j]}) ,
	\end{equation*}
	where we use that ${\rm d}_{\theta^{[j]}}  y^{[j]} =0.$ This concludes the proof.
\end{proof}

\begin{remark}
	Since $\sigma$ is applied componentwise, special caution needs to be exercised when deriving ${\rm d}_{y^{[i]}} a^{[i]}$. Let $r_j^{[i]}$ be the $j$th row of $W^{[i]} y^{[i]} +b^{[i]}$, namely $\sum_{m=1}^{n} W_{j,m}^{[i]} y_m^{[i]} + b_j^{[i]}$ for $j \in \{1,\ldots,n \}$. Then, with a slight abuse of notation, it holds
	\begin{align*}
		{\rm d}_{y^{[i]}} a^{[i]} = {\rm d}_{y^{[i]}} \sigma(W^{[i]} y^{[i]} + b^{[i]}) 
		&= {\rm d}_{y^{[i]}} \left( \sigma(r_j^{[i]}) \right)_{j=1}^{n} = \diag(\sigma'(r_1^{[i]}), \ldots,  \sigma'(r_{n}^{[i]})  ) \cdot W^{[i]},
	\end{align*}
	where $\sigma'(r_j^{[i]})$ is the one-dimensional derivative of $\sigma$ at $r_j^{[i]}$.
	Furthermore, for the partial derivative $\partial_{\theta^{[j]}} (\tau^{[j]} a^{[j]})$, we recall $\theta^{[j]} = (W^{[j]}(:), b^{[j]}, \tau^{[j]})^{\top} \in \mathbb{R}^N$, with $N = n^2 + n + 1$ and the fact that $a^{[j]}$ depends on $W^{[j]}$ and $b^{[j]}$, but not $\tau^{[j]}$. Consequently, we see
	\begin{equation*}
		\partial_{\theta^{[j]}} (\tau^{[j]} a^{[j]}) = \begin{pmatrix}
			\tau^{[j]} \partial_{W^{[j]}(:)} a^{[j]} &
			\tau^{[j]} \partial_{b^{[j]}} a^{[j]} &
			a^{[j]}
		\end{pmatrix} \in \mathbb{R}^{n \times N}.
	\end{equation*}
\end{remark}
As pointed out above, the standard feedforward neural network, where $\tau^{[\ell]}=1$ for all $\ell$, 
can suffer from vanishing and exploding gradients, which can be a challenge for optimization with deep networks 
\cite{bengio1994learning,glorot2010}. 
Consider the structure of the derivatives in  \eqref{eq:thetadiffvanilla} with $\tau^{[\ell]}=1$ for all $\ell$, 
e.g. for the final hidden layer with $\ell = L-1$, 
\begin{equation*}
	{\rm d}_{\theta^{[j]}}y^{[L-1]} =  \prod_{i=L-2}^{j+1} \left({\rm d}_{y^{[i]}}a^{[i]} \right)\partial_{\theta^{[j]}} a^{[j]}.
\end{equation*}
Especially in the one-dimensional case, it is obvious that if the partial derivatives ${\rm d}_{y^{[i]}}a^{[i]}$ are smaller than 1, the product will tend to 0 as the number of layers $L$ increases, which leads to vanishing gradients. On the other hand, if the partial derivatives ${\rm d}_{y^{[i]}}a^{[i]}$ are larger than 1, the product will tend to $\infty$ as the number of layers $L$ increases, which leads to exploding gradients. The feedforward neural network with variable~$\tau$, can potentially help overcome both problems, since now we have flexibility with respect to $\tau^{[\ell]}$. But one needs to be careful as if the gradient components are really small, then $\tau^{[\ell]}$ needs to be really large to compensate, which could lead to ill-conditioning issues.

A more appropriate approach is the standard ResNet with $\tau^{[\ell]}=\tau \in \R$ for all $\ell$. It is known to be stable with respect to vanishing gradients. Recalling the gradient from  \eqref{eq:thetadiffResNet} 
\begin{equation*}
	{\rm d}_{\theta^{[j]}}y^{[L-1]} =  \prod_{i=L-2}^{j+1 } \left(\mathbb{I} + \tau^{[i]} {\rm d}_{y^{[i]}}a^{[i]}   \right) \partial_{\theta^{[j]}} (\tau^{[j]}a^{[j]})  ,
\end{equation*}
it becomes clear that this stability is achieved by the added identity $\mathbb{I}$ in every part of the product. Hence, even if the Jacobians ${\rm d}_{y^{[i]}}a^{[i]}$ vanish, the product still contains the identity matrices. This advantage carries over to the $\tau$-variable framework.

Furthermore, the introduction of $\tau^{[\ell]}$ in ResNet allows us to tackle the exploding gradients problem. This property has also been discussed, using probabilistic bounds, in \cite{hayou2021stable}. Our approach is determinstic. The standard ResNet architecture does not have this property. Appropriate small $\tau^{[\ell]}$ can prevent the product from exploding with growing number of layers. However, choosing $\tau^{[\ell]}$ too small may lead to vanishing gradient problem again, as we will illustrate in the following simple example. Recall that, we do not tune $\tau^{[\ell]}$ by hand, but let the optimization find it.

\begin{example} \label{ex:ResNet}
	Consider the ResNet architecture with variable $\tau$ in one dimenstion, i.e. one node per layer. We have 
	\begin{align*}
		{\rm d}_{\theta^{[1]}} y^{[2]} &= \partial_{\theta^{[1]}} (\tau^{[1]} a^{[1]}), \\
		{\rm d}_{\theta^{[0]}} y^{[2]} &=  (1+ \tau^{[1]} {\rm d}_{y^{[1]}}a^{[1]} )\,  \partial_{\theta^{[0]}} (\tau^{[0]}a^{[0]}).
	\end{align*}
	Assume that ${\rm d}_{y^{[1]}}a^{[1]} $ is large, so that it leads to a large ${\rm d}_{\theta^{[0]}} y^{[2]}$.  
	This problem can be overcome, if $\tau^{[1]}$ attains a small value. However, this may lead to ${\rm d}_{\theta^{[1]}} y^{[2]} $ being accordingly small in its first two components, i.e. the derivatives by the weights and biases. Consequently, fixing one potential exploding gradient problem, can cause another gradient to vanish.
	However, as emphasized earlier, we do not tune $\tau^{[\ell]}$ by hand, but let the optimization find optimal values.
\end{example}
We also analyze the respective derivatives in the DenseNet architecture with 
variable $\tau$.
Finding a closed form for ${\rm d}_{\theta^{[j]}} y^{[\ell]}$ is not so easy for this network architecture, but we can derive
a recursive relation in terms of lower order terms. 

\begin{theorem}\label{thm:dense}
	Consider the DenseNet with $\tau$-variable framework
	\begin{equation*}
		y^{[\ell]} = \sum_{k=0}^{\ell -1} y^{[k]} + \tau^{[\ell-1]} a^{[\ell - 1]}, \qquad \ell = 1, \ldots, L-1. 
	\end{equation*} 
	Then the derivatives can be recursively written as 
	\begin{align*}
		{\rm d}_{\theta^{[j]}} y^{[i]} &= {\rm d}_{\theta^{[j]}} \sum_{k=j+1}^{i-2} y^{[k]}+ (\mathbb{I} + \tau^{[i-1]} {\rm d}_{y^{[i-1]}} a^{[i-1]}) \, {\rm d}_{\theta^{[j]}} y^{[i -1]} \qquad i = \ell, \ldots, j+2, \\
		{\rm d}_{\theta^{[j]}} y^{[j+1]} &=  \partial_{\theta^{[j]}} (\tau^{[j]}a^{[j]}).
	\end{align*}
\end{theorem}

\begin{proof}
	For $i=\ell,\ldots,j+2$ we employ the chain rule of differentiation and use that ${\rm d}_{\theta^{[j]}} y^{[k]} =0$ for $k < j+1$ to 
	arrive at
	\begin{align*}
		{\rm d}_{\theta^{[j]}} y^{[i]} &= {\rm d}_{\theta^{[j]}} \left( \sum_{k=0}^{i -1} y^{[k]} + \tau^{[ i-1 ]} a^{[i-1]} \right) \\
		&= {\rm d}_{\theta^{[j]}}\sum_{k=j+1}^{i -2} y^{[k]} +  {\rm d}_{\theta^{[j]}} y^{[i-1]} + \tau^{[i-1]} {\rm d}_{\theta^{[j]}} a^{[i-1]} \\
		&={\rm d}_{\theta^{[j]}} \sum_{k=j+1}^{i-2} y^{[k]}+ (\mathbb{I} + \tau^{[ i-1]} {\rm d}_{y^{[i-1]}} a^{[i-1]}) \, {\rm d}_{\theta^{[j]}} y^{[i -1]}.
	\end{align*}
	The case $i=j+1$ is special, since the chain rule does not need to be applied here. It simply holds 
	\begin{equation*}
		{\rm d}_{\theta^{[j]}} y^{[j+1]} =  {\rm d}_{\theta^{[j]}} \left( \sum_{k=0}^{j} y^{[k]} + \tau^{[j]} a^{[j]} \right) = 
		\partial_{\theta^{[j]}} (\tau^{[j]}a^{[j]}),
	\end{equation*}
	because $k < j+1$. The proof is complete.
\end{proof}

\begin{remark}
	In Theorem~\ref{thm:dense} we can successively insert the expressions for the lower order terms in the higher order terms, so that finally ${\rm d}_{\theta^{[j]}} y^{[\ell]}$ depends only on $\partial_{\theta^{[j]}} (\tau^{[j]}a^{[j]})$ and ${\rm d}_{y^{[i]}} a^{[i]}$ for $i=j+1,\ldots,\ell-1$. Furthermore, we see that every next lower order term enters with a factor $(\mathbb{I}+ \tau^{[i]} {\rm d}_{y^{[i]}} a^{[i]} )$, so that one can 
	overcome the vanishing gradients problem in a DenseNet (both fixed and variable $\tau$ cases). 
	To discuss the exploding gradients problem we consider for example the derivative ${\rm d}_{\theta^{[j]}} y^{[L-1]}$, 
	where one summand will be 
	$ \prod_{i=j+1}^{L-2} \left(\tau^{[i]} {\rm d}_{y^{[i]}} a^{[i]} \right) \partial_{\theta^{[j]}} (\tau^{[j]}a^{[j]}).$
	In the standard one-dimensional DenseNet architecture with $\tau^{[\ell]} =1$ for all $\ell$, we see that the above product tends to $\infty$ with 
	a growing number of layers $L$ if ${\rm d}_{y^{[i]}} a^{[i]} >1$ for all $i$. 
	The $\tau$-variable architecture can help deal with this problem, see also Example \ref{ex:Gradients}.
\end{remark}
 Similarly to the above cases, we can express derivatives of Fractional-DNN architecture, 
with variable $\tau$, in terms of lower order terms.
\begin{theorem}
	Consider the Fractional-DNN with $\tau$-variable framework
	\begin{equation*}
		y^{[\ell]} = y^{[\ell - 1]}  - \sum_{k=0}^{\ell-2} a_{\ell,k} (y^{[k+1]} - y^{[k]}) + (\tau^{[\ell -1]})^\gamma
		c_{\gamma-1}^{-1}
		a^{[\ell - 1]}, \qquad \ell = 1, \ldots, L-1 .
	\end{equation*}
	Then the derivatives can be recursively written as 
	\begin{align*}
		{\rm d}_{\theta^{[j]}} y^{[i]} &= {\rm d}_{\theta^{[j]}} \sum_{k=j+1}^{i-2} (a_{i,k} - a_{i,k-1}) y^{[k]} + \left( (1- a_{i,i-2})\mathbb{I} + (\tau^{[i-1]})^\gamma 
		c_{\gamma-1}^{-1}
		\, {\rm d}_{y^{[i-1]}}a^{[i-1]} \right) \, {\rm d}_{\theta^{[j]}} y^{[i -1]}, \\
		&\qquad\qquad\qquad\qquad  i = \ell, \ldots, j+2, \\
		{\rm d}_{\theta^{[j]}} y^{[j+1]} &=
		c_{\gamma-1}^{-1}  
		\,\partial_{\theta^{[j]}} ( (\tau^{[j]})^\gamma a^{[j]}).
	\end{align*}
\end{theorem}
 \begin{proof}
	First of all, we rewrite the forward propagation for $\ell=1,\ldots,L-1$ in Fractional-DNN
	\begin{align*}
		y^{[\ell]} &= y^{[\ell - 1]}  - \sum_{k=0}^{\ell-2} a_{\ell,k} (y^{[k+1]} - y^{[k]}) + (\tau^{[\ell -1]})^\gamma
		c_{\gamma-1}^{-1} 
		a^{[\ell - 1]}  \\
		&= a_{\ell,0} y^{[0]} + \sum_{k=1}^{\ell-2} ( a_{\ell,k} - a_{\ell,k-1}) y^{[k]} 
		+(1- a_{\ell,\ell-2}) y^{[\ell -1]} 
		+ (\tau^{[\ell -1]})^\gamma
		c_{\gamma-1}^{-1}  
		a^{[\ell - 1]}.
	\end{align*}
	Then for $i=\ell,\ldots,j+2$ we use chain rule and ${\rm d}_{\theta^{[j]}} y^{[k]} =0$ for $k < j+1$ to obtain 
	\begin{align*}
		{\rm d}_{\theta^{[j]}} y^{[i]} &= {\rm d}_{\theta^{[j]}} \Big( a_{i,0} y^{[0]} + \sum_{k=1}^{i-2} ( a_{i,k} - a_{i,k-1}) y^{[k]} 
		+(1- a_{i,i-2}) y^{[i -1]} 
		+ (\tau^{[i -1]})^\gamma 
		c_{\gamma-1}^{-1}
		a^{[i-1]} \Big) \\
		&={\rm d}_{\theta^{[j]}} \sum_{k=j+1}^{i-2} (a_{i,k} - a_{i,k-1}) y^{[k]} + \left( (1- a_{i,i-2})\mathbb{I} + (\tau^{[i-1]})^\gamma
		c_{\gamma-1}^{-1}  
		\, {\rm d}_{y^{[i-1]}}a^{[i-1]} \right) \, {\rm d}_{\theta^{[j]}} y^{[i -1]}. 
	\end{align*}
	Finally, for $i=j+1$, we exploit again ${\rm d}_{\theta^{[j]}} y^{[k]} =0$ for $k < j+1$, and derive
	\begin{align*}
		{\rm d}_{\theta^{[j]}} y^{[j+1]} &= {\rm d}_{\theta^{[j]}} \Big( a_{j+1,0} y^{[0]} + \sum_{k=1}^{j-1} ( a_{j+1,k} - a_{j+1,k-1}) y^{[k]} 
		+(1- a_{j+1,j-1}) y^{[j]} 
		+ (\tau^{[j]})^\gamma 
		c_{\gamma-1}^{-1} 
		a^{[j]} \Big)
		\\
		&=
		c_{\gamma-1}^{-1}  
		\,\partial_{\theta^{[j]}} ((\tau^{[j]})^\gamma  a^{[j]}).
	\end{align*}
	This completes the proof.
\end{proof} 
\begin{remark}
	Again, the lower order term representations can be successively inserted into the higher 
	order terms until we arrive at ${\rm d}_{\theta^{[j]}} y^{[\ell]}$ depending only on 
	$\partial_{\theta^{[j]}} ((\tau^{[j]})^\gamma a^{[j]})$ and ${\rm d}_{y^{[i]}} a^{[i]}$ for $i=j+1,\ldots,\ell-1$. Here, the next lower 
	order term enters with a factor 
	$	\big( (1- a_{i,i-2})\mathbb{I} + (\tau^{[i-1]})^\gamma 
	c_{\gamma-1}^{-1}  
	\, {\rm d}_{y^{[i-1]}}a^{[i-1]} \big),$ 
	which allows us to overcome the vanishing gradient problem in Fractional-DNNs.
	Furthermore, the multiplication by $(\tau^{[i-1]})^\gamma$ in this factor can help 
	us to deal with exploding gradients. This is similar to ResNet with variable $\tau$. 
\end{remark}
\begin{example} \label{ex:Gradients}
	To get an idea of how different network architectures influence the derivatives, the derivative ${\rm d}_{\theta^{[0]}} y^{[3]}$ 
	is displayed here for the four different options that have been considered in this section, i.e., feedforward neural network, 
	ResNet, DenseNet and Fractional DNN:
	\begin{align*}
		{\rm d}_{\theta^{[0]}} y^{[3]} &= \tau^{[2]} {\rm d}_{y^{[2]}} a^{[2]} \cdot \tau^{[1]} {\rm d}_{y^{[1]}} a^{[1]} \cdot	\partial_{\theta^{[0]}} (\tau^{[0]} a^{[0]}) , \\
		{\rm d}_{\theta^{[0]}} y^{[3]} &= \left(  \mathbb{I} + \tau^{[1]} {\rm d}_{y^{[1]}} a^{[1]} +  \tau^{[2]} {\rm d}_{y^{[2]}} a^{[2]} + \tau^{[1]} \tau^{[2]} {\rm d}_{y^{[2]}} a^{[2]} \cdot {\rm d}_{y^{[1]}} a^{[1]}  \right) \, 	\partial_{\theta^{[0]}} (\tau^{[0]} a^{[0]}) , \\
		{\rm d}_{\theta^{[0]}} y^{[3]} &= \left(  2 \mathbb{I} + \tau^{[1]} {\rm d}_{y^{[1]}} a^{[1]} + \tau^{[2]} {\rm d}_{y^{[2]}} a^{[2]} + \tau^{[1]}\tau^{[2]} {\rm d}_{y^{[2]}} a^{[2]} \cdot {\rm d}_{y^{[1]}} a^{[1]}  \right) \, 	\partial_{\theta^{[0]}} (\tau^{[0]} a^{[0]})  , \\
		{\rm d}_{\theta^{[0]}} y^{[3]} &=\{ (1 - a_{2,0} - a_{3,0} + a_{2,0}a_{3,1})\mathbb{I} + (1-a_{3,1})   (\tau^{[1]})^\gamma 
		c_{\gamma-1}^{-1}  
		{\rm d}_{y^{[1]}} a^{[1]} 
		\\
		&\quad + (1-a_{2,0}) (\tau^{[2]})^\gamma 
		c_{\gamma-1}^{-1} 
		{\rm d}_{y^{[2]}} a^{[2]} + (\tau^{[1]})^\gamma (\tau^{[2]})^\gamma 
		c_{\gamma-1}^{-2} 
		{\rm d}_{y^{[2]}} a^{[2]}  \cdot {\rm d}_{y^{[1]}} a^{[1]}    \} \; 
		c_{\gamma-1}^{-1} 
		\partial_{\theta^{[0]}} ( (\tau^{[0]})^\gamma  a^{[0]}). \\
	\end{align*}
\end{example}
In conclusion, ResNet, DenseNet and Fractional-DNN have a visible additive structure in the derivatives, which helps with the vanishing gradients problem. Furthermore, the parameters $\tau^{[\ell]}$ can help overcome both vanishing and exploding gradients.

\section{Numerical results} \label{sec:Num}

In this section, we apply the $\tau$-variable framework to a ResNet and a Fractional DNN
with and without bias ordering \eqref{eq:bias_order}. A thorough comparison is carried out in 
the context of an ill-posed 3D parametrized Maxwell's equation with Gauss's law. This 
problem is ill-posed because the standard N\'ed\'elec finite element is only curl conforming 
and cannot directly impose the Gauss's law.
%
In all the cases, we apply the smoothed version of standard 
$\text{ReLU}(y)= \max\{ 0,y \} $ as the activation function
\begin{equation*}
	\text{smoothReLU}(y) = \begin{cases}
		\max\{0,y \}, & \text{if } |y| > \eta \\
		\frac{1}{4 \eta} y^2 + 0.5 y + 0.25 \eta, &\text{if } \; y \in [-\eta,\eta] \, . 
	\end{cases}
\end{equation*}
We have found that $\eta = 10^{-4}$ is a robust
choice for the examples under consideration.
Notice, that one can also use other activation functions which can 
differ from layer to layer.

\subsection{Maxwell's equations} \label{subsec:Maxwells}

Our findings suggests that the $\tau$-variable framework outperforms the standard approach (with fixed $\tau$) for deeper networks,
see Figure \ref{fig:mse_6_50}. This is expected, since the 
effect of variable $\tau^{[\ell]}$ will 
be more prominent when more layers (and consequently more time-step parameters $\tau^{[\ell]}$) are present. On the other hand, for shallow networks, the $\tau$-variable
framework provides comparatively less improvements, see Figure~\ref{fig:c1}. Nevertheless, the
$\tau$-variable framework applied to ResNet yields error improvements compared to a standard ResNet for model extrapolation, see Figure \ref{fig:PWError}. These results
are also comparable to the approximation obtained with the finite element method (FEM) with  the lowest order N\'ed\'elec space, cf. Figure \ref{fig:d1}.

Consider the 
Maxwell-Dirac equations. 
Our goal is to learn $\bm u:\Omega\subset \R^3\mapsto \R^3$  that satisfies 
\begin{align}
\begin{aligned}
 \mathrm{curl}\left( \bm \mu^{-1} \mathrm{curl } \bm u\right)  &= \bm f  \quad \mbox{ in } \Omega,\\
 \mathrm{div} \!\left(\bm \varepsilon \bm u\right)&=\rho \quad \mbox{ in } \Omega,\\
 \bm u\times \bm n &= \bm g \quad \mbox{ on } \partial \Omega,
 \end{aligned} \label{def:StationaryMax}
\end{align}
where $\bm \mu $ and $\bm \varepsilon$ are positive definite symmetric tensors in $L^\infty(\Omega)^3 $, $\bm f\in L^2(\Omega)^3,$ $\rho \in L^2(\Omega)$ and  $\bm g \in H^{-\frac 12}_{||}(\mathrm{div}_\Gamma;\partial \Omega)$.
This problem is particularly dif{}ficult at the discrete level due to its divergence-related constraints and requires
 rather tailored algorithms to deal with it,  see for instance \cite{Ciarlet2014}.
Therefore, an interesting question is to approximate the map:
\begin{align*}
(\bm x,\bm f(\bm x),\bm \mu(\bm x),\rho(\bm x))  \mapsto \bm u(\bm x), 
\end{align*}
that can lead to a reasonable and noise-robust approximation of the solution $\bm u$ to \eqref{def:StationaryMax}, 
note that we ignore the boundary data $\bm g$. 
This approach is similar to a surrogate model where its output could be used as  an initial guess by an iterative method like the domain decomposition method or in the reduced basis method \cite{TTNQuyen_HAntil_HDiaz_2022a}. Nevertheless, those problems are beyond the scope of the present paper, and it will be studied in future works.

This learning problem is challenging because the solutions to \eqref{def:StationaryMax}
can have discontinuities, while most neural networks, except for the ones with the Heaviside activation function, lead to continuous approximations.
Also, it is still not clear how to incorporate the geometry (domain $\Omega$) of the problem in a meaningful way.
Thus, we consider an example with a known smooth solution and we compare it with an approximation obtained by various DNNs and by the lowest order N\'ed\'elec space of the f{}irst kind, cf. \cite{Nedelec1980}, denoted by $\mathcal{N}_0(\Omega)$.
 Here, we consider  the basis proposed in \cite{Jay2005}.  In order to do that, let us consider  $\bm \varepsilon=\mathbb{I}_{3\times 3}$,  and for a smooth $\varphi:\Omega\mapsto \R^+$ we def{}ine $ \bm \mu^{-1}(\bm x)= \varphi(\bm x) \mathbb{I}_{3\times 3},$  then
$
     \mathrm{curl}\left(   \bm \mu^{-1} \mathrm{curl} \bm u \right)=\nabla \varphi\times \mathrm{curl} \bm u  + \varphi \mathrm{curl}\left(\mathrm{curl} \bm u \right).
$

Thus, if we consider  
\begin{align}
\bm u: \Omega&\mapsto \R^3, \quad 
(x_1,x_2,x_3)\mapsto  I_1(r(x_1,x_2,x_3)) \bm e_\theta, \label{def:exact_u}\\
\varphi:\Omega&\mapsto \R, \quad (x_1,x_2,x_3)\mapsto \frac{1}{2}(x_1^2+x_2^2+1), \nonumber
\end{align}
 where $\Omega$ is the cylinder  $\{(x_1,x_2,x_3)\in \R^3\!: x_1^2+x_2^2 \leq 1 \mbox{ and } x_3\in [0,1] \}$, $I_\nu$ is the  modif{}ied Bessel functions of the f{}irst kind of order $\nu,$
 $r(x_1,x_2,x_3)= \sqrt{x_1^2+x_2^2}$, and $\bm e_\theta(x_1,x_2,x_3)=(x_1^2+x_2^2)^{-\frac 12}(-x_2,x_1,0),$ 
we obtain: 
 \begin{align*}
    \mathrm{curl} \bm u = I_0(r)\bm e_z, \quad 
     \mathrm{curl} \left( \mathrm{curl} \bm u \right)&=  - \bm u, \quad 
     \mathrm{div} \bm u = 0, \mbox{ and }\\
      \mathrm{curl}\left(   \bm \mu^{-1} \mathrm{curl} \bm u \right)&= -rI_0(r)\bm e_\theta-\varphi \bm u=:\bm f.
 \end{align*} 

Because $\bm u$ is divergence free, we consider the reduced map     
$(\bm x,\bm f(\bm x),\bm \varphi(\bm x))  \mapsto \bm u(\bm x),$
where $\bm x = (x_1,x_2,x_3),$ and $\bm u(\bm x)=(\bm u_1(\bm x),\bm u_2(\bm x),\bm u_3(\bm x)).$ 
In order to generate the input/output data for the DNNs, we consider points $\{\bm x_i\}_{i=1}^N\subset \Omega$ randomly chosen from  $\Omega$ obtained with Matlab's function $\mathrm{unifrnd}$ along with $\mathrm{philox}$  as the random number algorithm. Here, we set $N=12,000$. Then, $\{(\bm x_i,f(\bm x_i),\varphi(\bm x_i) ) \}_{i=1}^N $ and  $\{\bm u(\bm x_i)\}_{i=1}^N$ can be utilized as input/output data.
 We are now ready to train and compare several DNNs.

\subsection*{Neural Network size \& network reduction} 
The bigger the network that we use for training, the bigger the computational time and the memory requirements. We 
employ the $\tau$-variable framework and start with a ResNet architecture with 
5 hidden layers with 10 nodes each. As target functional we implement the mean squared error with no regularization, i.e. $\lambda_1 = \lambda_2 =0$. Additionally, we consider bias ordering (B.O.) with a fixed Moreau-Yosida parameter $\beta=10$. After 1000 steepest descent steps we observe the following result:
	The relative error in the Euclidean norm on the test set is 0.07, and we see $\tau^{[1]},\tau^{[3]}$ and $\tau^{[4]}$ are approximately 0. Recalling the ResNet structure with variable $\tau$, \eqref{eq:OVarRes}, it is obvious that e.g. from $\tau^{[1]}\approx 0$ we can deduce $y^{[2]} \approx P_1^2 y^{[1]}$. Consequently, we delete the hidden layers 2,4 and 5, cf. Figure \ref{fig:reduce}. The reduced network with 2 hidden layers achieves the same relative error on the test set, i.e. 0.07. 
\begin{figure}[h]
	\includegraphics[width=0.49\textwidth]{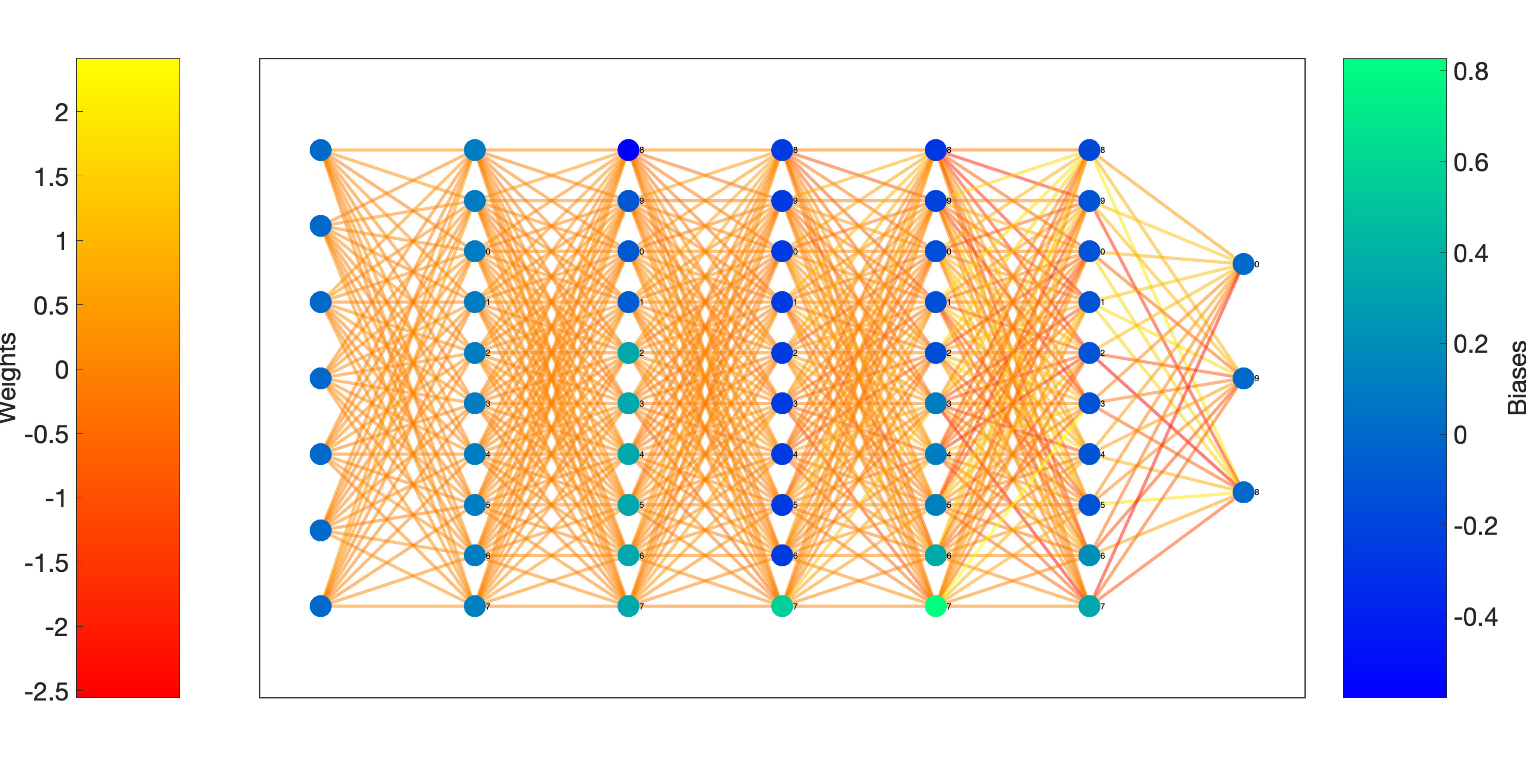} 
	\includegraphics[width=0.49\textwidth]{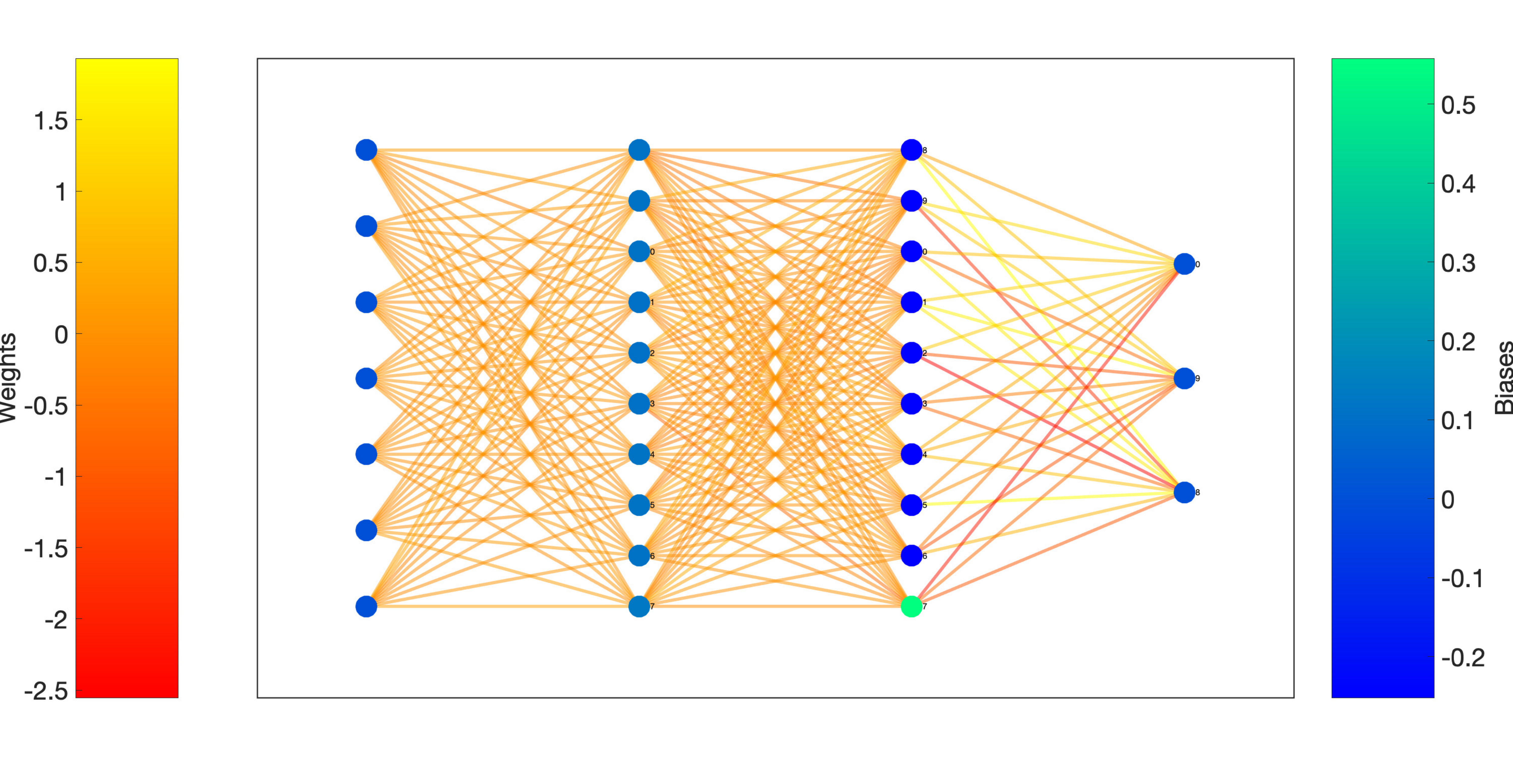}
	
	\caption{Left: Optimal weights and biases for ResNet with variable $\tau$ with 5 hidden layers and 10 nodes each with bias ordering. Right: Reduced ResNet with 2 hidden layers, i.e. hidden layers 1 and 3 from the larger network. The color of the dots indicates the bias value and the color of the lines indicates the magnitude of the weight.}
	
	\label{fig:reduce}
\end{figure}

While the same relative error is obtained with the reduced network, we aim at achieving even better results, therefore we
next consider a larger network size with 6 hidden layers and 50 nodes in each layer (6-50).
The proposed variable-$\tau$
approach seems to always outperform its constant $\tau$ counterpart, see 
Figure \ref{fig:mse_6_50}. As stated before, this is expected because the variable-$\tau$ framework and also Fractional-DNN have a bigger impact for deeper architectures with more hidden layers. Let us remark that the curves in Figure \ref{fig:mse_6_50} are not monotone because we have only plotted the mean squared error term. In case of the entire $J$ we do observe monotone behavior as expected. 
Even though we see in Table~\ref{tab:taus} that $\tau^{[\ell]} > 0$ for all $\ell$ in this setup, 
motivated by the reduction in the previous architecture of 5 hidden layers and 10 nodes per layer, instead of (6-50), we also consider a network with 2 hidden layers with 50 nodes per layer (2-50), which yields better results, cf. Figure \ref{fig:SumFEM}. 
\subsection*{Results: 6 layers-50 nodes vs. 2 layers-50 nodes } 
From now on, $\bm u^{N\!N}(\bm x)$ will  denote the approximation of $\bm u$ obtained with a neural network at a point $\bm x$, the specif{}ic architecture will be clear from the context. Note that, $\bm u^{N\!N}(\bm x)=(\bm u_1^{N\!N}(\bm x),\bm u_2^{N\!N}(\bm x),\bm u_3^{N\!N}(\bm x))\in \R^3.$  

\begin{figure}[h!]
\centering  
\subfigure[Fixed $\tau$ -- 6 layers and 50 nodes]{\label{fig:a1}\includegraphics[width=0.46\textwidth]{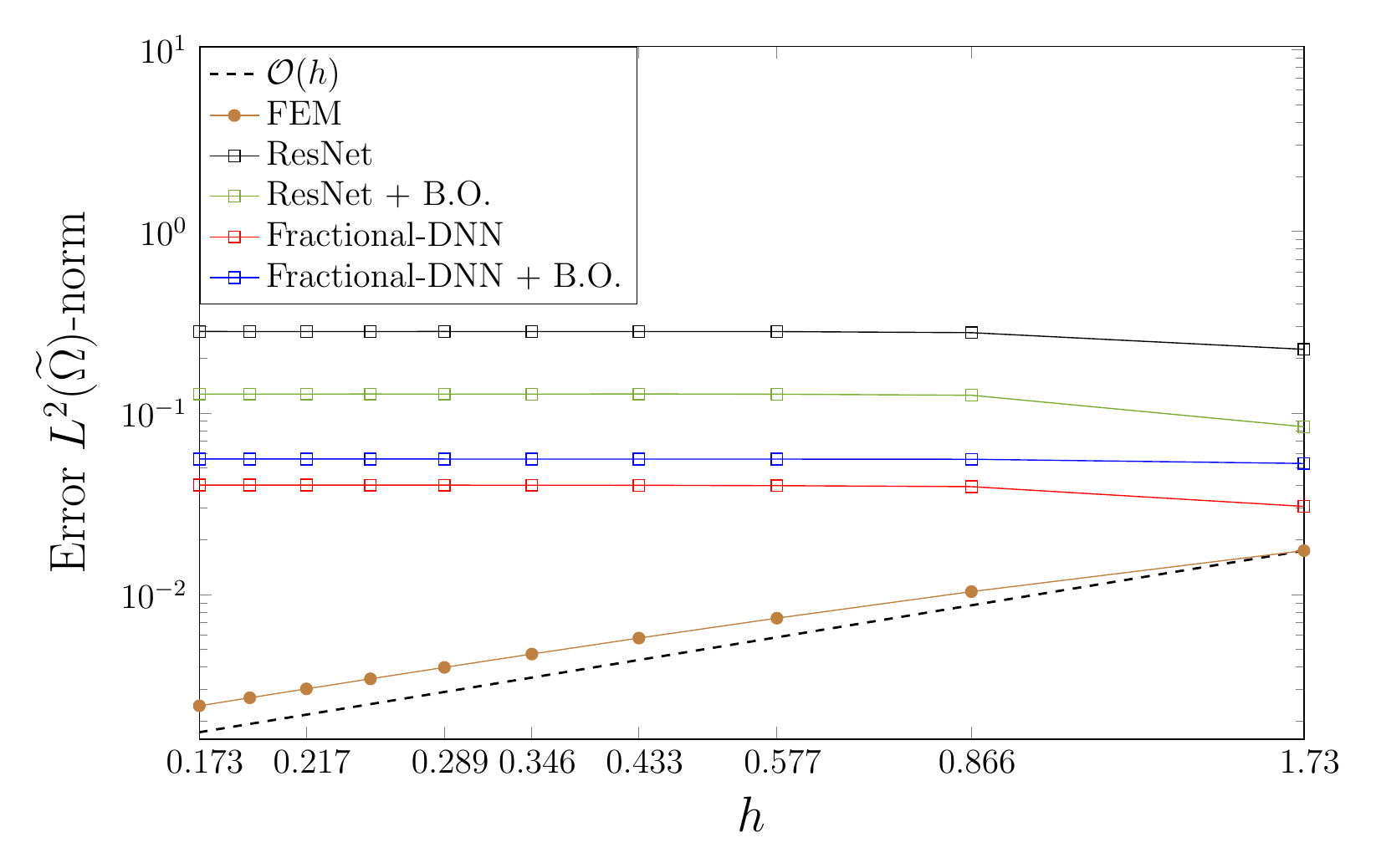}}
\subfigure[{ Variable $\tau$ -- 6 layers and 50 nodes}]{\label{fig:b1} \includegraphics[width=0.46\textwidth]{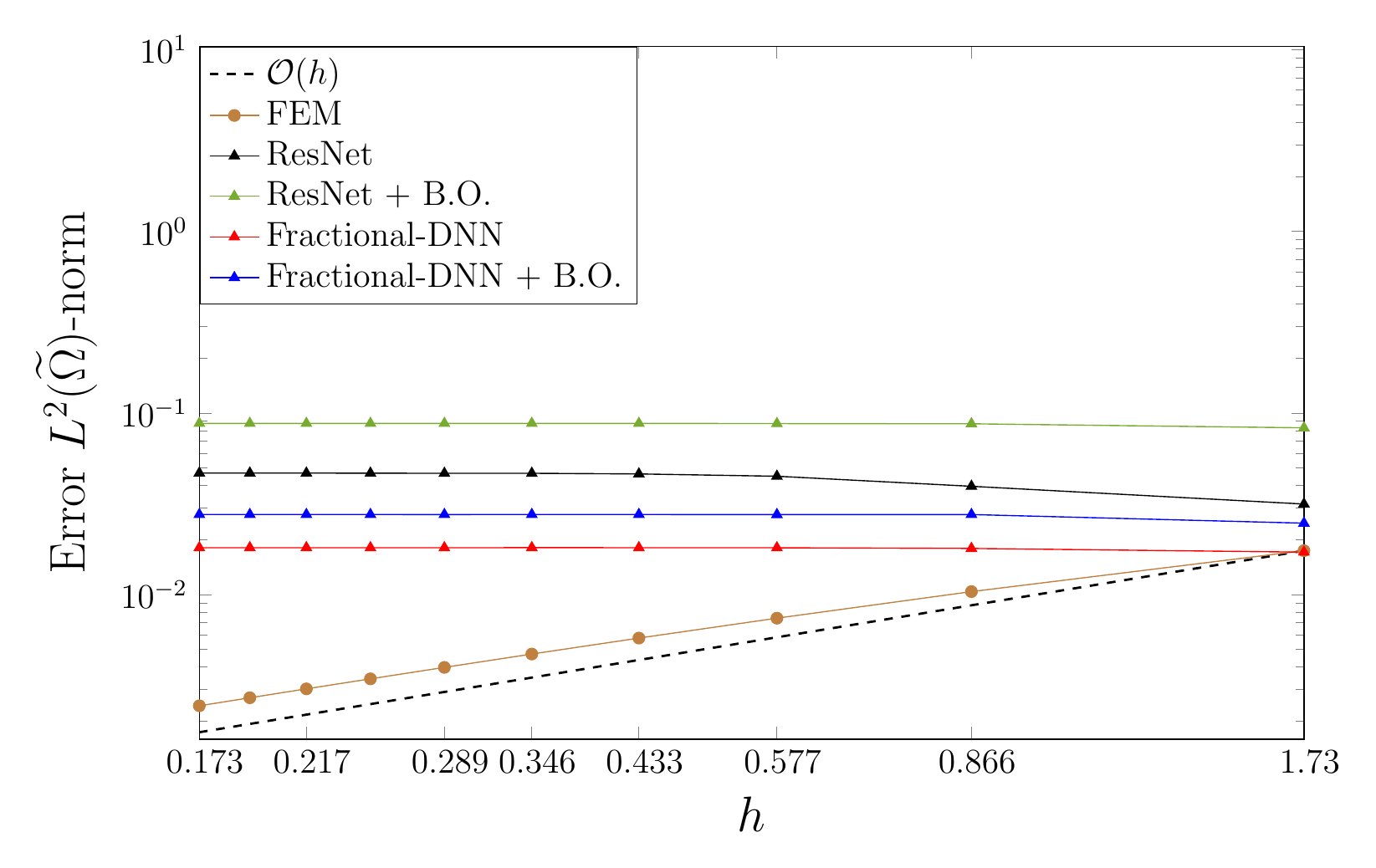} }

\hspace{-0.3cm}
\subfigure[MSE plot --  2 layers and 50 nodes]{\label{fig:c1}  \includegraphics[width=0.48\textwidth]{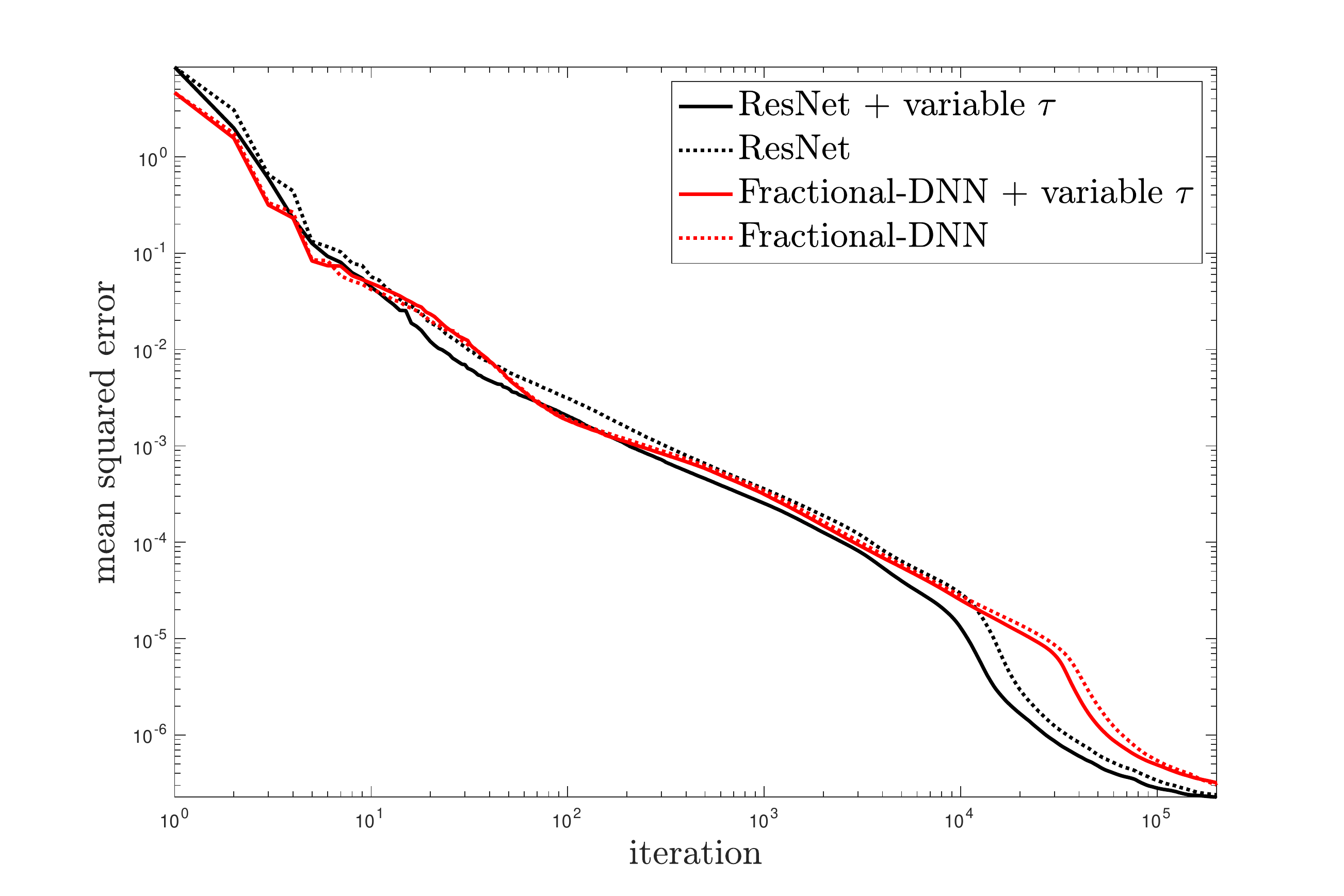} }   
\hspace{-0.75cm}
\subfigure[Comparison  -- 2 layers and 50 nodes]{\label{fig:d1} \includegraphics[width=0.46\textwidth]{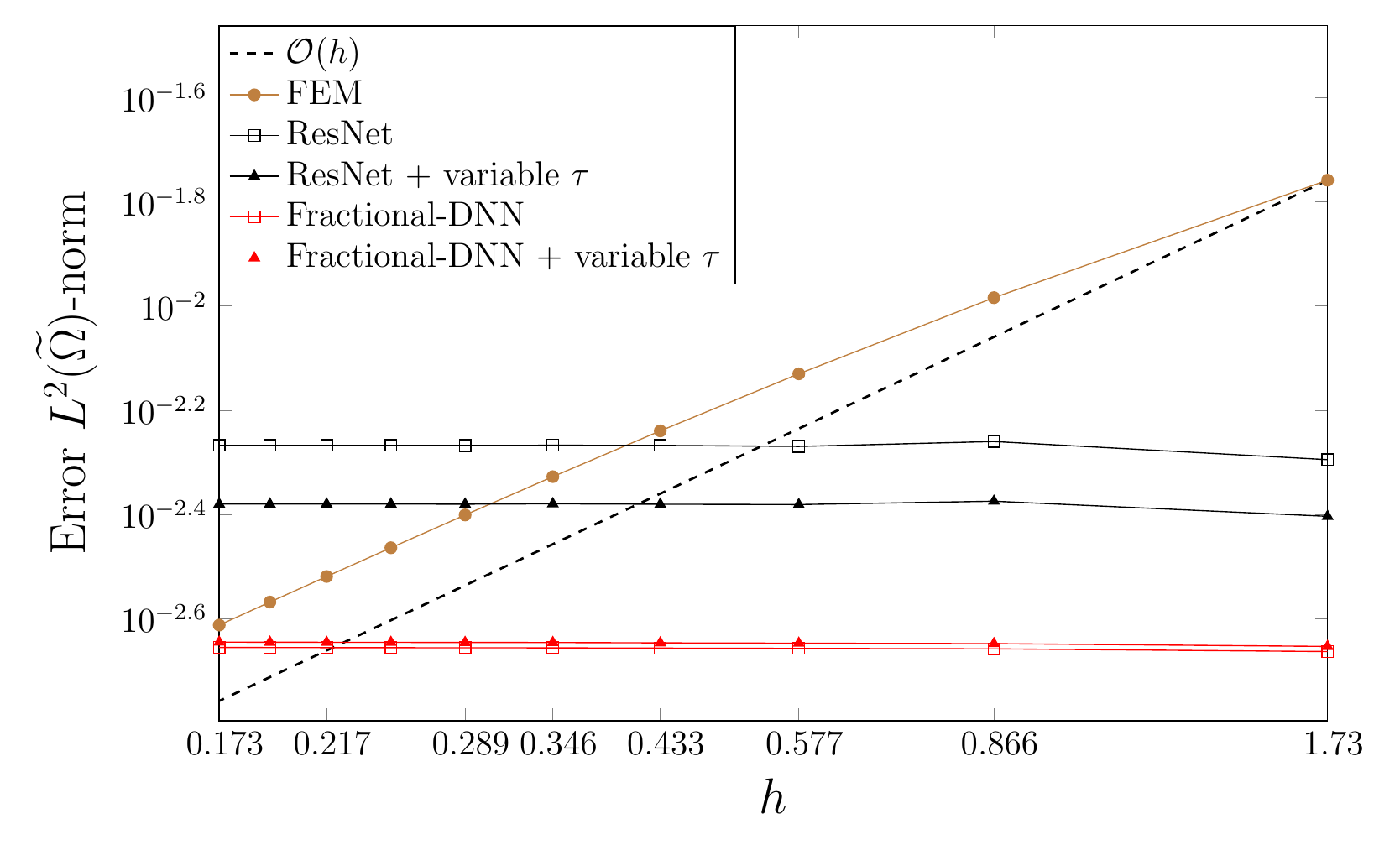} }
     \caption{Comparison between various DNN architectures and FEM. \textbf{Top row:} $L^2$-error between an exact solution and DNN approximation (6 hidden layers with a width of 50 each) or FEM approximation. B.O. indicates bias ordering. The left and right panels correspond to fixed and variable $\tau$, respectively. \textbf{Bottom row:} The left panel shows the mean squared error during training of different DNNs with 2 hidden layers with a width of 50 nodes each. The right panel displays the $L^2$-error between an exact solution and DNN approximation for the same DNNs and FEM.}
     \label{fig:SumFEM}
\end{figure}

\begin{table}[h]
	\begin{tabular}{| l | c | c | c | c | c | c || l | c | c |}
		\hline
		\textbf{6-50 }& $\tau^{[0]}$ & $\tau^{[1]}$ & $\tau^{[2]}$ & $\tau^{[3]}$ & $\tau^{[4]}$ & $\tau^{[5]}$ & \textbf{2-50} & $\tau^{[0]}$ & $\tau^{[1]}$  \\
		\hline
		ResNet & 0.92 & 0.95 & 0.99 & 0.95 & 0.92 & 0.88 & & 0.84 & 0.87 \\
		\hline
		ResNet + B.O. & 0.70 & 0.94 & 1.00 & 0.96 & 0.86 & 0.70 & & 0.28 & 0.51 \\
		\hline
		Fractional-DNN & 0.59 & 0.70 & 0.53 & 0.34 & 0.28 & 0.30 & & 0.94 & 0.93 \\
		\hline
		Fractional-DNN + B.O. & 0.67 & 0.80 & 0.78 & 0.68 & 0.44 & 0.04 & & 0.85 & 0.95 \\
		\hline
	\end{tabular}
	\caption{Optimal learned $\tau$ variables for various DNN architectures with $\tau$-variable framework with 6 layers and 2 layers. These are the same network architectures that are considered in Figure~\ref{fig:SumFEM}.}
	\label{tab:taus}
\end{table}

We compare the neural network results with an approximation obtained with the FEM, see Figure~\ref{fig:a1},~\ref{fig:b1}, and \ref{fig:d1}. To do that, we consider the unit 
cube $(0,1)^3$, denoted by $\widetilde\Omega$ 
as a domain. The unit cube is considered, to test how well the Neural 
Network performs for 
unseen data, and to test its extrapolation properties. Recall that the training data
has been generated on $\Omega$, which is cylindrical.
For the FEM, we consider 10 uniform refinements of the unit cube $(0,1)^3$ and denote by $h$ the mesh size of each one. Then, we compute  $\|\bm u- \bm u^{N\!N}  \|_{\widetilde\Omega}$ and $\|\bm u- \bm u_h \|_{\widetilde\Omega}$, where $\|\cdot \|_{\widetilde\Omega}$ denotes the $L^2(\widetilde\Omega)^3-$norm and $\bm u_h$ denotes the best approximation of $\bm u$ into $ \mathcal{N}_0(\widetilde\Omega)$, with respect to the $H(\mathrm{curl};\widetilde\Omega)-$norm.

In Figure~\ref{fig:d1}, we observe that, for the DNN with 2 layers and for large $h$, the DNN approach gives a better approximation than the Lowest Order N\'ed\'elec space. We further notice that, as $h$ gets smaller, $\bm u^{N\!N}$ needs to be evaluated at more points in $\widetilde\Omega \setminus \Omega$ and it is not obvious if the DNN approximation will remain stable. However, Figure~\ref{fig:a1},~\ref{fig:b1}, and \ref{fig:d1} show that the DNN approximation remains stable.

\begin{figure}[h!]
\centering     
\subfigure[ResNet]{\label{fig:Rf}\includegraphics[scale=0.32]{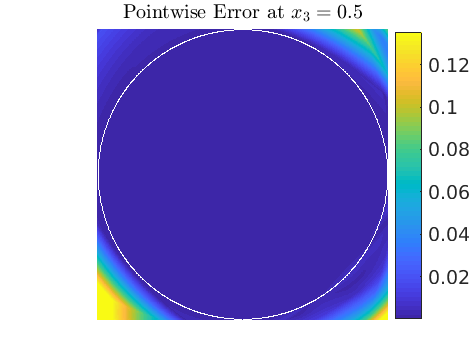}}
\subfigure[ResNet + $\tau$]{\label{fig:Rt}\includegraphics[scale=0.32]{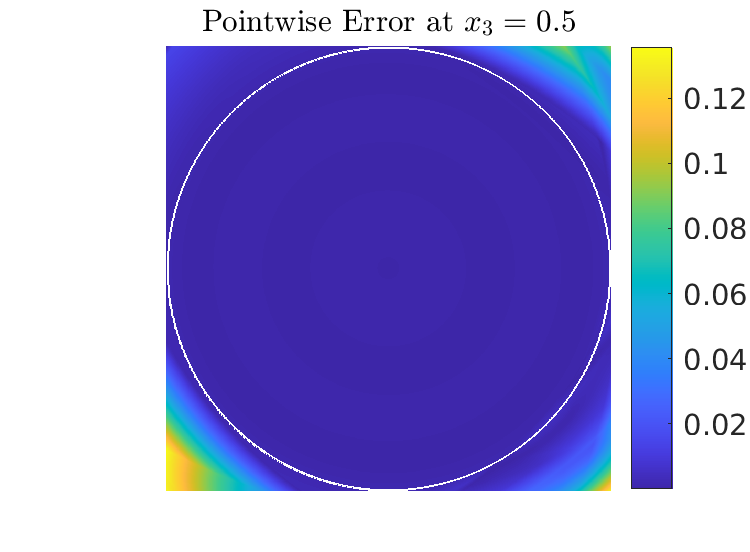}}
\subfigure[Fractional-DNN]{\label{fig:Ff}\includegraphics[scale=0.32]{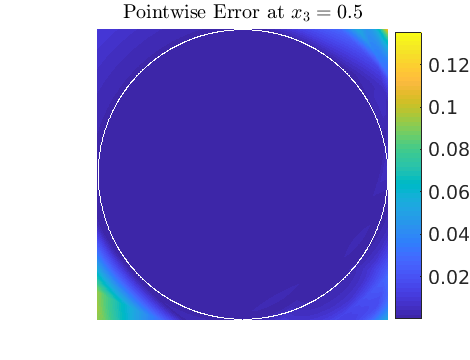}}
\subfigure[Fractional-DNN + $\tau$]{\label{fig:Ft}\includegraphics[scale=0.32]{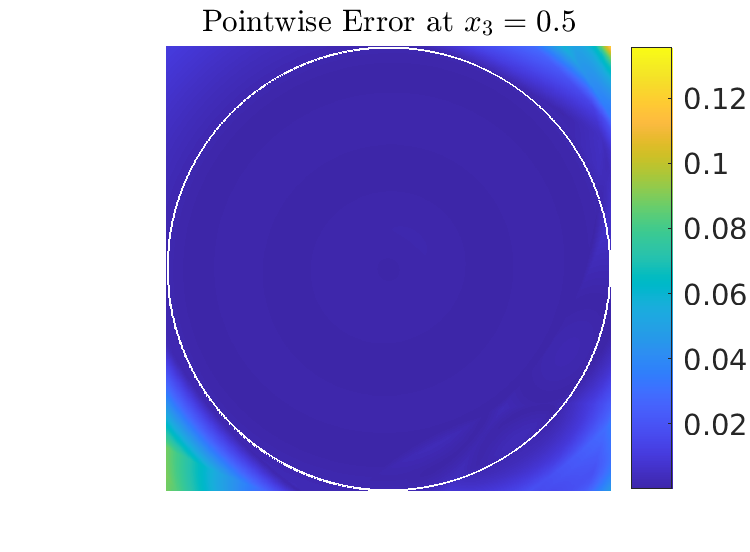}}

\caption{Comparison of testing errors between ResNet, ResNet with $\tau$-learning framework, Fractional-DNN and Fractional-DNN with $\tau$-learning framework (2-50). 
}
\label{fig:PWError} 
\end{figure}

\begin{figure}[h!]
 \centering
   \begin{tikzpicture}
     \node at (-7, 1) {\includegraphics[scale=0.43]{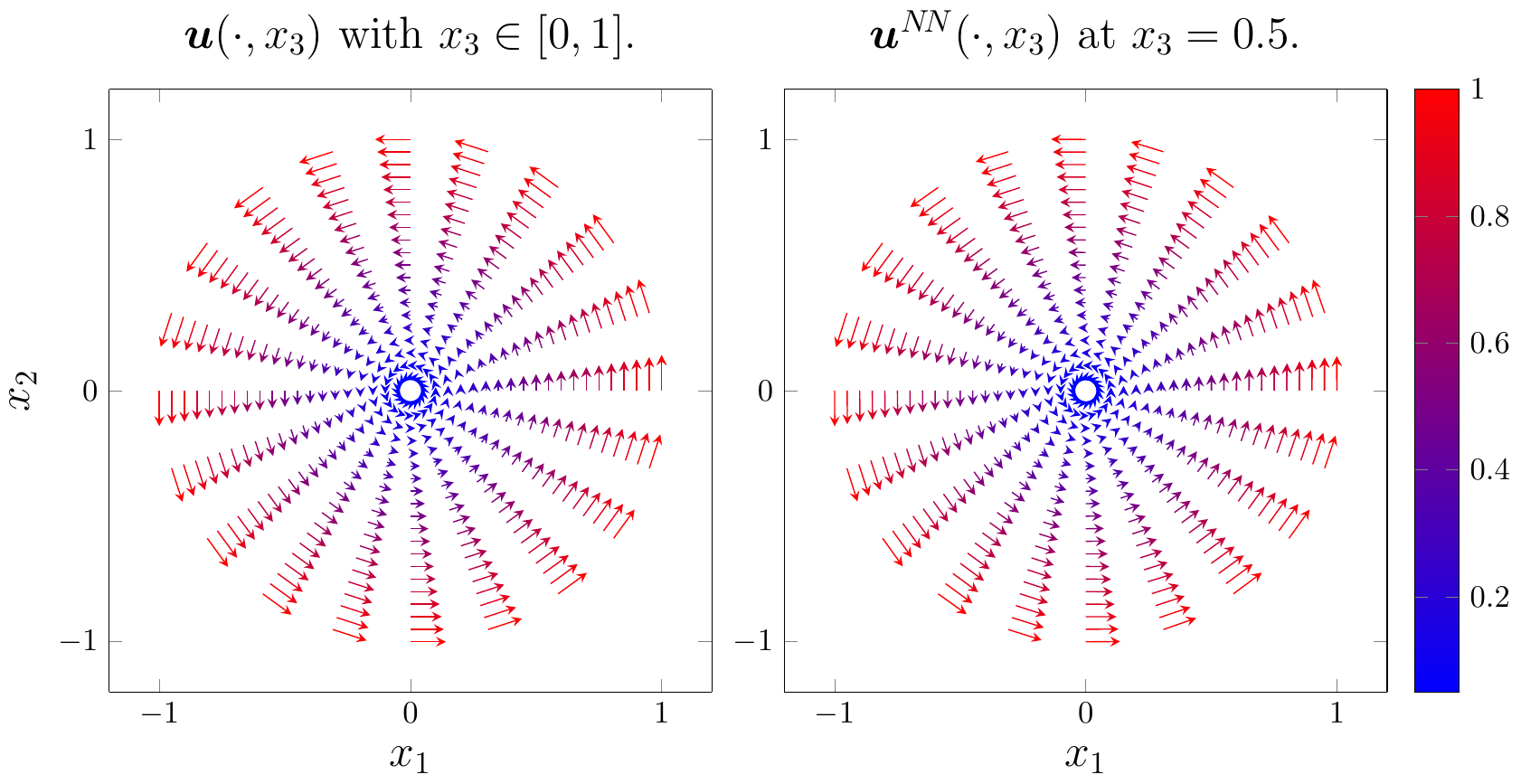}};
     \node at (-0.7, 1) {\includegraphics[scale=0.33]{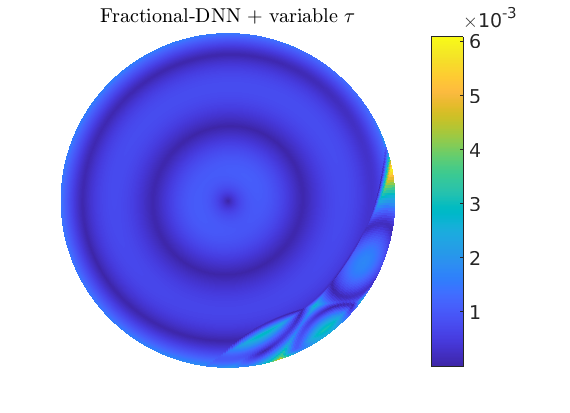}};
          \node at (3.8, 1) {\includegraphics[scale=0.33]{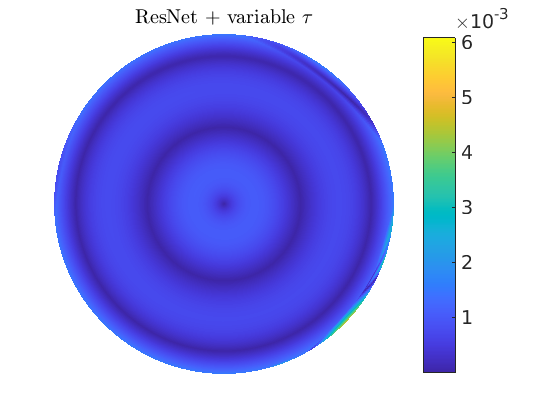}};
  \end{tikzpicture}
\caption{  $\bm u,$ $\bm u^{N\!N}$ and pointwise error  on $\Omega$, at $x_3 = 0.5$  ($x_1x_2$-plane).}\label{fig:QuiverPlot}
\end{figure}

There exist several ways to measure how well a neural network performs. For instance, for the case of 2 layers and 50 nodes, the training error is slightly better with the ResNet-based architectures, cf. Figure \ref{fig:c1}, while a smaller error on unseen data is achieved with Fractional-DNNs,  cf. Figure \ref{fig:d1}. Furthermore, when we plot the error on the  square $(-1,1)^2\times \{0.5\}$, we see that Fractional-DNNs, cf. Figure \ref{fig:Ff} and \ref{fig:Ft}, extrapolate better than ResNets, cf. Figure \ref{fig:Rf} and \ref{fig:Rt}. In the ResNet setting, employing the $\tau$-learning framework (Figure \ref{fig:Rt}) yields to slightly better results than fixing $\tau$ (Figure \ref{fig:Rf}). 
Additionally, from its definition, we know $ \bm u_3\equiv 0,$  cf. \eqref{def:exact_u}. Hence, we present quiver plots of the first two components of $\bm u(\cdot,x_3)$ for any $x_3$, and $\bm u^{N\!N}(\cdot,0.5)$, in Figure~\ref{fig:QuiverPlot}. The two plots seem to coincide. Therefore, we further present pointwise errors restricted to $\Omega$, where the pointwise error is measured in the ($\R^2$) Euclidean norm. For	Fractional-DNN with variable $\tau$ (2-50) with no bias ordering, inside $\Omega$, we observe that 
$ -0.0051\leq \bm u^{N\!N}_{3}(\bm x)\leq 0.0077$, i.e., the constraints violation
is of the order of approximation error. Meanwhile, ResNet with variable $\tau$ achieves better results in this test case.

Besides improving the training, the approximation
could be further improved if we know a priori some qualitative properties of the exact 
solution. Then, they could be forced in the loss functional, similarly, as it is done with 
PINNs, cf. \cite{PINNs21}. It is important to mention that several other numerical 
examples were considered to test the robustness of our $\tau$-variable framework. 
For instance, we considered problems where the standard Neural ODEs 
(cf.~\cite{neuralOdes}) struggle to obtain good approximations, as pointed out in 
\cite{dupont2019}. We obtained similar results to the ones presented here. For the
sake of brevity, those results have been excluded.

\section*{Conclusion} \label{sec:Concl}
A time variable learning framework for DNNs has been introduced, which in general can be applied to any DNN. However, from a mathematical perspective, DNN architectures which can be related to dynamical systems are of interest, since learning $\tau$ then corresponds to optimal adaptive time stepping. Consequently, special emphasis has been put on applying the $\tau$-variable framework to ResNet and Fractional-DNN. The $\tau$-variable framework is argued to overcome vanishing and exploding gradient challenges. 
The numerical results suggest that DNNs with $\tau$-variable framework outperform their counterparts with fixed $\tau$ for deep architectures and enjoy an improved training error decay. Moreover, this method has the potential of identifying redundant layers, so that the network size can be reduced while maintaining the quality of the prediction.

\bibliographystyle{plain}
\bibliography{NR}

\appendix

\section{Derivatives of the Lagrangian $\mathcal{L}$} \label{appA}
We provide detailed calculations of derivatives needed in Section \ref{subsec:learningfrac}. To this end we recall
\begin{align*}
	\mathcal{L}(y,\theta,\phi) &=
	J (\theta)
	- \sum_{\ell=1}^{L-1} \langle y^{[\ell]} - P_{\ell-1}^{\ell}y^{[\ell-1]} + \sum_{j=0}^{\ell-2} a_{\ell-1,j} (P_{j+1}^{\ell}y^{[j+1]}- P_{j}^{\ell}y^{[j]}) \\
	&\qquad \qquad  - (\tau^{[\ell-1]})^\gamma 
	c_{\gamma-1}^{-1} 
	 \sigma ( W^{[\ell-1]} y^{[\ell-1]} + b^{[\ell-1]} )  , \phi^{[\ell]} \rangle \\ 
	&\qquad - \left\langle y^{[L]} - W^{[L-1]} y^{[L-1]} , \phi^{[L]} \right\rangle
\end{align*}

\subsection{Derivative with respect to $y^{[\ell]}$}
\label{appA1}
Here, we calculate the variation of $\mathcal{L}$ with respect to $y^{[\ell]}$ for $\ell=1,\ldots,L$:
\begin{align*}
\partial_{y^{[\ell]}}	\mathcal{L}(y,\theta,\phi) &= \partial_{y^{[\ell]}} J(\theta) - \partial_{y^{[\ell]}} \left( \sum_{k=1}^{L-1}  \langle y^{[k]} - P_{k-1}^{k}y^{[k-1]},  \phi^{[k]} \rangle \right)\\
&\quad
 - \partial_{y^{[\ell]}} \left(  \sum_{k=1}^{L-1} \left\langle \sum_{j=0}^{k-2} a_{k-1,j}(P_{j+1}^{k}y^{[j+1]} -P_{j}^{k}y^{[j]}), \phi^{[k]}  \right\rangle \right)\\
 &\quad 
 + \partial_{y^{[\ell]}} \left( \sum_{k=1}^{L-1} \left\langle(\tau^{[k-1]})^\gamma 
c_{\gamma-1}^{-1} 
 \sigma ( W^{[k-1]} y^{[k-1]} + b^{[k-1]} )  , \phi^{[k]} \right\rangle \right) \\
 &\quad 
 - \partial_{y^{[\ell]}} \left\langle y^{[L]} - W^{[L-1]} y^{[L-1]} , \phi^{[L]} \right\rangle.
\end{align*}
From e.g. \cite[Section 4.2.]{antil2020fractional} we have most components of this expression already given. The main difference lies in the factors $a_{k-1,j}$, since the contained $\tau^{[j]},\ldots,\tau^{[k-1]}$ are variable now. We only calculate the remaining unknown term, i.e. the second line in the above equation. First we rewrite the double sum in the following way: 
\begin{align*}
	&-   \sum_{k=1}^{L-1} \left\langle \sum_{j=0}^{k-2} a_{k-1,j}(P_{j+1}^{k}y^{[j+1]} -P_{j}^{k}y^{[j]}), \phi^{[k]}  \right\rangle \\
	&\quad =  \sum_{k=0}^{L-2} \left( \sum_{j=0}^{k-1} a_{k,j} \langle P_{j}^{k+1}y^{[j]}, \phi^{[k+1]}\rangle - \sum_{j=0}^{k-1} a_{k,j} \langle P_{j+1}^{k+1}y^{[j+1]}, \phi^{[k+1]}\rangle \right) \\
	&\quad = \sum_{k=0}^{L-2} \left( \sum_{j=0}^{k-1} a_{k,j} \langle P_{j}^{k+1}y^{[j]}, \phi^{[k+1]}\rangle - \sum_{j=1}^{k} a_{k,j-1} \langle P_{j}^{k+1}y^{[j]}, \phi^{[k+1]}\rangle \right).
\end{align*}
Now we take the derivative with respect to $y^{[\ell]}$ and can apply the sum rule of differentiation:
\begin{align*}
 &\sum_{k=0}^{L-2} \sum_{j=0}^{k-1} a_{k,j} \, \partial_{y^{[\ell]}}   \langle y^{[j]}, (P_{j}^{k+1})^{\top}\phi^{[k+1]}\rangle - \sum_{k=0}^{L-2} \sum_{j=1}^{k} a_{k,j-1}  \, \partial_{y^{[\ell]}} \langle y^{[j]}, (P_{j}^{k+1})^{\top} \phi^{[k+1]}\rangle \\
	&\quad =  \sum_{k=0, \ell \le k-1}^{L-2}  a_{k,\ell} (P_{\ell}^{k+1})^{\top} \phi^{[k+1]}- \sum_{k=0, \ell \le k}^{L-2}a_{k,\ell-1} (P_{\ell}^{k+1})^{\top}\phi^{[k+1]} \\
	&\quad = \sum_{k = \ell + 1}^{L-2} a_{k,\ell} (P_{\ell}^{k+1})^{\top}\phi^{[k+1]} - \sum_{k= \ell}^{L-2} a_{k,\ell-1} (P_{\ell}^{k+1})^{\top}\phi^{[k+1]} \\
\end{align*}
For $\ell = L$ and $\ell = L-1$ this derivative vanishes. For $\ell = 1,\ldots,L-2$ we can slightly rewrite to get  
\begin{align*}
	&\sum_{k = \ell + 1}^{L-2} a_{k,\ell} (P_{\ell}^{k+1})^{\top}\phi^{[k+1]} - \sum_{k= \ell}^{L-2} a_{k,\ell-1} (P_{\ell}^{k+1})^{\top}\phi^{[k+1]}  \\
	&\quad = \sum_{k = \ell + 2}^{L-1} (a_{k-1,\ell}-a_{k-1,\ell-1} ) (P_{\ell}^{k})^{\top}\phi^{[k]} - a_{\ell,\ell-1} (P_{\ell}^{\ell+1})^{\top}\phi^{[\ell +1]}.
\end{align*}
Now, the derivative can be assembled.

\subsection{Derivative with respect to $\tau^{[\ell]}$}
\label{appA2}

Care must be observed as $a_{k,j}$ contains $\tau^{[j]},\ldots,\tau^{[k]}$. 
The derivative of $\mathcal{L}(y,\theta,\phi)$ with respect to $\tau^{[\ell]}$ for $\ell=0,\ldots,L-2$ therefore consists of the following terms: 
\begin{align*}
	\partial_{\tau^{[\ell]}} \mathcal{L}(y,\theta,\phi) &= \partial_{\tau^{[\ell]}}
	J (\theta)
	- \partial_{\tau^{[\ell]}} \left( \sum_{k=1}^{L-1} \left\langle  \sum_{j=0}^{k-2} a_{k-1,j} (P_{j+1}^{k}y^{[j+1]}- P_{j}^{k}y^{[j]}), \phi^{[k]} \right\rangle \right) \\
	&\quad  +  \partial_{\tau^{[\ell]}} \left( \sum_{k=1}^{L-1}  \left\langle (\tau^{[k-1]})^\gamma 
	c_{\gamma-1}^{-1} 
	 \sigma ( W^{[k-1]} y^{[k-1]} + b^{[k-1]} )  , \phi^{[k]} \right\rangle \right).
\end{align*}
We make an index shift from $k-1$ to $k$, use the sum rule of differentiation and the fact that only $\tau^{[j]},\ldots,\tau^{[k]}$ are contained in $a_{k,j}$ to obtain 
\begin{align*}
	&\partial_{\tau^{[\ell]}} \left( -\sum_{k=1}^{L-1} \left\langle \sum_{j=0}^{k-2} a_{k-1,j} ({P_{j+1}^{k}}y^{[j+1]} - {P_{j}^{k}}y^{[j]}), \phi^{[k]} \right\rangle \right) \\
	&\quad = -\sum_{k=0}^{L-2}  \sum_{j=0}^{k-1} \partial_{\tau^{[\ell]}} (a_{k,j}) \left\langle {P_{j+1}^{k+1}}y^{[j+1]} - {P_{j}^{k+1}}y^{[j]}, \phi^{[k+1]} \right\rangle \\
	&\quad =-\sum_{k= \ell}^{L-2}  \sum_{j=0 }^{\min\{k-1,\ell\}} \partial_{\tau^{[\ell]}} (a_{k,j}) \left\langle {P_{j+1}^{k+1}}y^{[j+1]} - {P_{j}^{k+1}}y^{[j]}, \phi^{[k+1]} \right\rangle.
\end{align*}
Using the above equality, we arrive at
\begin{align*}
	\partial_{\tau^{[\ell]}} \mathcal{L}(y,\theta,\phi)   &= \partial_{\tau^{[\ell]}} J(\theta)   -\sum_{k= \ell}^{L-2}  \sum_{j=0 }^{\min\{k-1,\ell\}} \partial_{\tau^{[\ell]}} (a_{k,j}) \left\langle {P_{j+1}^{k+1}}y^{[j+1]} - {P_{j}^{k+1}}y^{[j]}, \phi^{[k+1]} \right\rangle  \\ 
&\quad+ \left\langle  \gamma (\tau^{[\ell]})^{\gamma-1}
c_{\gamma-1}^{-1}
 \sigma (W^{[\ell]} y^{[\ell]} + b^{[\ell]}), \phi^{[\ell+1]} \right\rangle.
\end{align*}
For implementations we may want to understand the double summation in more detail. In fact, it can be split up into 3 terms in the following way:
\begin{align*}
	&-\sum_{k= \ell}^{L-2}  \sum_{j=0 }^{\min\{k-1,\ell\}} \partial_{\tau^{[\ell]}} (a_{k,j}) \left\langle {P_{j+1}^{k+1}}y^{[j+1]} - {P_{j}^{k+1}}y^{[j]}, \phi^{[k+1]} \right\rangle  \\
	&\quad = -  \sum_{j=0 }^{\ell-1} \partial_{\tau^{[\ell]}} (a_{\ell,j}) \left\langle {P_{j+1}^{\ell+1}}y^{[j+1]} - {P_{j}^{\ell+1}}y^{[j]}, \phi^{[\ell+1]} \right\rangle \\
	&\qquad - \sum_{k= \ell+1}^{L-2}  \sum_{j=0 }^{\ell-1} \partial_{\tau^{[\ell]}} (a_{k,j}) \left\langle {P_{j+1}^{k+1}}y^{[j+1]} - {P_{j}^{k+1}}y^{[j]}, \phi^{[k+1]} \right\rangle \\ 
	&\qquad - \sum_{k= \ell+1}^{L-2} \partial_{\tau^{[\ell]}} (a_{k,\ell}) \left\langle {P_{\ell+1}^{k+1}}y^{[\ell+1]} - {P_{\ell}^{k+1}}y^{[\ell]}, \phi^{[k+1]} \right\rangle.
\end{align*}
Now for each of the above terms we can easily compute the contained derivative using basic differentiation rules.
Employing the product rule for $j=0,\ldots,\ell-1$, i.e. $j<\ell$, we get 
\begin{align*}
	 \partial_{\tau^{[\ell]}} (a_{\ell,j}) &= \textstyle (1-\gamma) \frac{(\tau^{[\ell]})^{\gamma}}{\tau^{[j]}} \left( \left( \sum_{i=j}^\ell \tau^{[i]} \right)^{-\gamma} - \left( \sum_{i=j+1}^\ell \tau^{[i]} \right)^{-\gamma}  \right) \\
	 &\quad + \textstyle\gamma \frac{(\tau^{[\ell]})^{\gamma-1}}{\tau^{[j]}} \left( \left( \sum_{i=j}^\ell \tau^{[i]} \right)^{1-\gamma} - \left( \sum_{i=j+1}^\ell \tau^{[i]} \right)^{1-\gamma}  \right) .
\end{align*}
For $j=0,\ldots,\ell-1$ and $k = \ell+1,\ldots,L-2$, i.e. $j < \ell < k$, we have 
\begin{align*}
	\partial_{\tau^{[\ell]}} (a_{k,j})  = \textstyle (1-\gamma) \frac{(\tau^{[k]})^{\gamma}}{\tau^{[j]}} \left( \left( \sum_{i=j}^k \tau^{[i]} \right)^{-\gamma} - \left( \sum_{i=j+1}^k \tau^{[i]} \right)^{-\gamma}  \right) .
\end{align*}
And finally, using the product rule again for $k = \ell+1,\ldots,L-2$, i.e. $\ell <k$, we see 
\begin{align*}
	 \partial_{\tau^{[\ell]}} (a_{k,\ell}) &= \textstyle \frac{(\tau^{[k]})^{\gamma}}{(\tau^{[\ell]})^2} \left(  \left( \sum_{i=\ell+1}^k \tau^{[i]} \right)^{1-\gamma} - \left( \sum_{i=\ell}^k \tau^{[i]} \right)^{1-\gamma}\right) + (1-\gamma) \frac{(\tau^{[k]})^{\gamma}}{\tau^{[\ell]}}  \left( \sum_{i=\ell}^k \tau^{[i]} \right)^{-\gamma}
\end{align*}

\end{document}